\documentclass{article}
\usepackage[utf8]{inputenc} \usepackage[T1]{fontenc}    \usepackage[english]{babel}

\usepackage{amsfonts}
\usepackage{nicefrac}       \usepackage{microtype}      \usepackage{lmodern}
\usepackage{amssymb,amsmath,amsthm}
\usepackage{stmaryrd}
\usepackage{bbm,bm}
\usepackage{latexsym}
\usepackage{xcolor}

\usepackage{enumerate}
\usepackage{verbatim}
\usepackage{booktabs}       

\usepackage{url}            \usepackage[colorlinks=true]{hyperref}
\usepackage[numbers,sort&compress,square,comma]{natbib}
\usepackage[capitalise,noabbrev]{cleveref}

\usepackage{parskip}
\usepackage{geometry}
\usepackage[short]{optidef}
\usepackage{algorithmic,algorithm}
\usepackage{authblk}

\usepackage{thmtools}

\declaretheorem{theorem}
\declaretheorem{corollary}
\declaretheorem{lemma}
\declaretheorem{proposition}

\declaretheoremstyle[qed=$\square$]{definitionwithend}
\declaretheorem[style=definitionwithend]{definition}
\declaretheorem[style=definitionwithend]{assumption}
\declaretheorem[style=definitionwithend]{example}
\declaretheorem[style=definitionwithend]{remark}

\crefname{fact}{Fact}{Facts}
\crefname{assumption}{Assumption}{Assumptions}

\usepackage{import}
\usepackage{pdfpages}
\usepackage{transparent}

\definecolor{gold}{rgb}{0.85,0.65,0}

\newcommand{\abs}[1]{\ensuremath{\left\lvert #1 \right\rvert}}

\newcommand{\by}{\times}
 
\newcommand{\norm}[1]{\ensuremath{\left\lVert #1 \right\rVert}}
\newcommand{\ip}[1]{\ensuremath{\left\langle #1 \right\rangle}}

\let\emptyset\varnothing
\newcommand{\set}[1]{\left\{#1\right\}}

\newcommand{\bb}{\mathbb}

\def\A{{\mathbb{A}}}

\def\C{{\mathbb{C}}}

\def\L{{\mathbb{L}}}

\def\R{{\mathbb{R}}}
\def\S{{\mathbb{S}}}

\def\cA{{\cal A}}
\def\cB{{\cal B}}

\def\cD{{\cal D}}
\def\cE{{\cal E}}
\def\cF{{\cal F}}
\def\cG{{\cal G}}
\def\cH{{\cal H}}

\def\cJ{{\cal J}}

\def\cM{{\cal M}}
\def\cN{{\cal N}}

\def\cR{{\cal R}}
\def\cS{{\cal S}}
\def\cT{{\cal T}}

\def\cV{{\cal V}}

\def\cX{{\cal X}}

\DeclareMathOperator{\Opt}{Opt}
\DeclareMathOperator*{\argmin}{arg\,min}

\DeclareMathOperator{\rank}{rank}
\DeclareMathOperator{\Diag}{Diag}
\DeclareMathOperator{\diag}{diag}
\DeclareMathOperator{\tr}{tr}

\DeclareMathOperator{\range}{range}
\newenvironment{smallpmatrix}
{\left(\begin{smallmatrix}}
{\end{smallmatrix}\right)}

\DeclareMathOperator*{\E}{\mathbb{E}}

\DeclareMathOperator{\aff}{aff}
\DeclareMathOperator{\spann}{span}
\DeclareMathOperator{\inter}{int}
\DeclareMathOperator{\rint}{rint}
\DeclareMathOperator{\cl}{cl}

\DeclareMathOperator{\rbd}{rbd}

\DeclareMathOperator{\conv}{conv}
\DeclareMathOperator{\cone}{cone}
\DeclareMathOperator{\clconv}{clconv}
\DeclareMathOperator{\clcone}{clcone}
\def\extr{{\mathop{\rm extr}}}

 \graphicspath{{figures/}}

\usepackage{thm-restate}

\usepackage{tikz-cd}
\usepackage{tikz}
\usepackage[export]{adjustbox}

\usetikzlibrary{angles,quotes}
\DeclareMathOperator{\Sym}{Sym}

\newcommand{\obj}{\textup{obj}}
\newcommand{\gobj}{\gamma_\obj}
\newcommand{\qobj}{q_\obj}
\newcommand{\Mobj}{M_\obj}

\PassOptionsToPackage{numbers, compress}{natbib}
\renewcommand{\cite}{\citet}

\usepackage{soul}

\definecolor{lightblue}{HTML}{00afb9}

\begin{document}

\title{Exactness in SDP relaxations of QCQPs: Theory and applications}
\author[]{Fatma K{\i}l{\i}n\c{c}-Karzan}
\author[]{Alex L.\ Wang}
\affil[]{\texttt{\normalsize\{fkilinc,alw1\}@andrew.cmu.edu}	}
\affil[]{Carnegie Mellon University, Pittsburgh, PA, 15213, USA.}

\date{\today}

\maketitle
\begin{abstract}
Quadratically constrained quadratic programs (QCQPs) are a fundamental class of optimization problems. In a QCQP, we are asked to minimize a (possibly nonconvex) quadratic function subject to a number of (possibly nonconvex) quadratic constraints. Such problems arise naturally in many areas of operations research, computer science, and engineering. Although QCQPs are NP-hard to solve in general, they admit a natural convex relaxation via the standard (Shor) semidefinite program (SDP) relaxation. 
In this tutorial, we will study the SDP relaxation for general QCQPs, present various exactness concepts related to this relaxation and discuss conditions guaranteeing such SDP exactness.
In particular, we will define and examine three notions of SDP exactness: (i) \textit{objective value exactness}---the condition that the optimal value of the QCQP and the optimal value of its SDP relaxation coincide,  (ii) \textit{convex hull exactness}---the condition that the convex hull of the QCQP epigraph coincides with the (projected) SDP epigraph, and (iii) the \emph{rank-one generated} (ROG) property---the condition that a particular conic subset of the positive semidefinite matrices related to a given QCQP is generated by its rank-one matrices.
Our analysis for objective value exactness and convex hull exactness stems from a geometric treatment of the projected SDP relaxation and crucially considers how the objective function interacts with the constraints.
The ROG property complements these results by offering a sufficient condition for both objective value exactness and convex hull exactness which is oblivious to the objective function.
By analyzing the geometry of the associated sets, we will give a variety of sufficient conditions for these exactness conditions and discuss settings where these sufficient conditions are additionally necessary.
Throughout, we will highlight implications of our results for a number of example applications.

 \end{abstract}
\section{Introduction}\label{sec:Intro}

Quadratically constrained quadratic programs (QCQPs) are a fundamental class of nonconvex optimization problems that naturally arise in operations research, engineering, and computer science; see \cite{wang2021tightness} for additional applications of QCQPs.
The ubiquity of this class of optimization problems stems from its expressiveness; any $\set{0,1}$ integer program or polynomial optimization problem may be recast as a QCQP (see~\cite{phan1982quadratically,bao2011semidefinite,benTal2009robust} and references therein).

It is well known that QCQPs are NP-hard to solve in general---indeed, the NP-hard combinatorial problem Max-Cut can be readily recast as a QCQP.
On the other hand, the standard (Shor) semidefinite program (SDP) relaxation offers a natural tractable convex relaxation for a general QCQP \cite{shor1990dual}. This convex relaxation is obtained by first reformulating the QCQP in a lifted space with an additional rank constraint and then dropping the rank constraint. 
Several papers have studied the quality of this relaxation for specific problem classes such as Max-Cut \cite{goemans1995improved} as well as for more general QCQPs \cite{ye1999approximating,nesterov1997quality,benTal2001lectures,megretski2001relaxations}.

\citet{laurent1995positive} showed that it is NP-hard to determine whether the SDP relaxation of a given QCQP has \textit{objective value exactness}, i.e., whether the optimum objective value of the QCQP matches the optimum objective value of its SDP relaxation. Nevertheless, a recent line of work has focused on \emph{sufficient} conditions that ensure various forms of SDP exactness. In this direction, prior work has focused on the case where there are only a few (usually one or two) nonconvex quadratic functions in the QCQP. This research vein can be traced back to Yakubovich's S-procedure \cite{yakubovich1971s,fradkov1979s-procedure} (also known as the S-lemma) and the work of \citet{sturm2003cones}. In particular, the classical trust region subproblem (TRS)---the problem of minimizing a nonconvex quadratic function over an ellipsoid---and its variants have attracted significant attention and cases under which an exact SDP reformulation is possible have been investigated; see the excellent survey by \citet{burer2015gentle} and references therein.
For example,
\citet{jeyakumar2013trust} showed that the standard SDP relaxation of the TRS with additional linear inequalities has objective value exactness under a condition regarding the dimension of the minimum generalized eigenspace.
\citet{hoNguyen2017second} give a generalization of \cite[Section 6]{jeyakumar2013trust} (see \citet[Section 2.2]{hoNguyen2017second} for a comparison of these conditions). Follow-up work by \citet{wang2020generalized} extends these results to the setting of the generalized trust region subproblem (GTRS)---the problem of minimizing a nonconvex quadratic function over a nonconvex quadratic constraint.
Additional work \cite{benTal2014hidden,locatelli2015some,locatelli2016exactness} observes that under a simultaneous diagonalizability assumption, it is possible to rewrite the SDP relaxation as a second-order cone program (SOCP). Sufficient conditions for objective value exactness can then be derived by analyzing the KKT-multipliers of the associated SOCP. This approach is investigated for the extended TRS \cite{locatelli2015some} and simultaneously diagonalizable QCQPs with two constraints \cite{benTal2014hidden}.

A more recent line of research has focused on sufficient conditions for different forms of exactness which do not make explicit assumptions on the number of nonconvex quadratic functions.
In this context, objective value exactness has been investigated frequently. As an example, \citet{burer2019exact,locatelli2020kkt} establish sufficient conditions under which diagonal QCQPs (those QCQPs with diagonal quadratic forms) have this property.
\citet{wang2021tightness} study the geometry of the set of convex Lagrange (dual) multipliers and a natural symmetry parameter of the QCQP and establish that objective value exactness (among other forms of exactness) holds whenever the set of Lagrange multipliers is polyhedral and the symmetry parameter is large enough.
On the one hand, the framework presented in \cite{wang2021tightness} is general enough to cover and extend many existing results \cite{locatelli2015some,burer2019exact,fradkov1979s-procedure,hoNguyen2017second,wang2020generalized} on objective value exactness.
On the other hand, the embedded assumption that the set of convex Lagrange multipliers is polyhedral is restrictive and prevents the results in \cite{wang2021tightness} from being applied to a number of interesting QCQPs (including some that are known to have exact SDP relaxations). 
More recently, \citet{wang2020geometric} presented a generalization of their framework from \cite{wang2021tightness} and record sufficient conditions for objective value exactness for \emph{general} sets of convex Lagrange multipliers.
They furthermore show that these sufficient conditions are also necessary in the context of \emph{convex hull exactness} (see the following paragraph) when the polar cone to the set of convex Lagrange multipliers is facially exposed.

\citet{wang2020convex,wang2021tightness,wang2020geometric} complement sufficient conditions for objective value exactness with sufficient conditions
for \textit{convex hull exactness}, i.e., the property that the convex hull of the QCQP epigraph coincides with the (projected) SDP epigraph.
Although convex hull exactness is a stronger property than objective value exactness, it is also more widely applicable.
Specifically, results showing how to convexify commonly occurring substructures in nonconvex problems are useful in building strong convex relaxations for more complicated nonconvex problems.
Such results have advanced state-of-the-art computational approaches for mixed integer linear programs and general nonlinear nonconvex programs; see \cite{conforti2014integer,tawarmalani2002convexification}. 
In this direction, \cite{wang2020convex,wang2021tightness} established that the convex hull of the epigraph of a ``highly-symmetric'' QCQP with a polyhedral set of convex Lagrange multipliers is given by its projected SDP relaxation.
These results were further generalized in \cite{wang2020geometric} where the assumption that the set of convex Lagrange multipliers is polyhedral is dropped.
The results from \cite{wang2020convex,wang2021tightness} recover a number of existing results \cite{hoNguyen2017second,wang2020generalized,modaresi2017convex,yildiran2009convex,burer2017how}, including convex hull characterizations of specific sets defined by one or two quadratics.
The generalization given in \cite{wang2020geometric} widely broadens the applicability of the framework and recovers convex hull results for quadratic matrix programs \cite{beck2007quadratic,beck2012new} as well as a basic mixed-binary set related to the ``perspective reformulation/relaxation trick'' \cite{gunluk2010perspective,frangioni2006perspective}.
We remark that the ``perspective reformulation/relaxation trick'' is well known in the literature and has been useful in deriving convex hull exactness for a variety of sets arising in sparsity-constrained optimization \cite{gunluk2010perspective,ceria1999convex,frangioni2006perspective,wei2020convexification,wei2020ideal,atamturk2018sparse,dong2013valid}.

The sufficient conditions for objective value and convex hull exactness presented in \cite{wang2021tightness,wang2020geometric,wang2020convex} rely heavily (unsurprisingly) on how the objective function interacts with the constraints.
A separate line of work complements these results by investigating conditions for SDP exactness which are oblivious to the objective function. Such results are useful, for example, in settings where the objective function may evolve with time or are not known a priori.
In this direction, an important geometric property, namely, the \emph{rank-one generated} (ROG) property (first coined in \cite{hildebrand2016spectrahedral}), plays an important role. A closed convex conic subset of the positive semidefinite cone is said to be ROG if it is equal to the convex hull of its rank-one matrices. 
As an immediate example, the positive semidefinite cone itself is ROG.
One may compare the ROG property of a closed conic subset of the positive semidefinite cone with the integrality property of a polytope. In both cases, the property states that the convex set in question is the convex hull of a nonconvex set of interest.

In contrast to well-known sufficient conditions, e.g., total unimodularity or total dual integrality, for the integrality property of polyhedra (see \cite{deCarliSilva2020notion} and references therein),
the research on the ROG property is much more recent and limited.
Indeed, this property until recently had only been studied \emph{incidentally} to other research questions.
A series of works in the matrix completion literature \cite{grone1984positive,agler1988positive,paulsen1989schur}
show that the set of positive semidefinite matrices with a fixed chordal support is ROG.
The celebrated S-lemma \cite{yakubovich1971s} (see also \cite{sturm2003cones}) can be interpreted as saying that the intersection of the positive semidefinite cone with a single linear matrix inequality (LMI) is ROG.
A closely related line of work gives explicit descriptions of the ROG cones related to quadratic programs over triangles, tetrahedra, and quadrilaterals \cite{anstreicher2010computable}, and ellipsoids missing caps \cite{burer2014trust}; see also the excellent survey paper \cite{burer2015gentle}.
More recently,
\citet{hildebrand2016spectrahedral,blekherman2017sums} study algebraic properties of ROG cones obtained by adding homogeneous linear matrix equalities (LMEs) to the positive semidefinite cone.

Extending this line of work, \citet{argue2020necessary} examine the question of what the ROG property of a conic subset of the positive semidefinite cone corresponds to in terms of its defining LMIs.
They present a toolset for studying the ROG property based on facial structure (similar to that for integrality in polyhedra) and use it to derive a number of sufficient conditions for the property. This toolset is additionally used to give an explicit characterization of the ROG cones defined by at most two LMIs. This extends one of the few settings in the literature---the case of a single LMI and the S-lemma---where an explicit characterization of the ROG property is understood.

\subsection{Overview and outline}

This tutorial summarizes some recent work on notions of exactness in SDP relaxations of QCQPs: objective value exactness, convex hull exactness, and the ROG property. 
Specifically, the contents of this tutorial track closely with results first presented in \cite{wang2021tightness,wang2020geometric,argue2020necessary,wang2020convex}.
An outline of this tutorial is as follows:
\begin{enumerate}[(a)]

\item We begin with preliminaries in \cref{sec:Prelim}. \cref{sec:Setup,sec:projected_SDP} introduce QCQPs and their SDP relaxations.
\cref{sec:aggregation} then states our main assumptions and recalls basic facts from the framework of \cite{wang2021tightness, wang2020geometric}.
Specifically, we recall the definition of the set of convex Lagrange multipliers and its role in defining the SDP relaxation (see \cref{lem:cS_SDP_description}).
This description is the starting point for both the objective value and convex hull exactness results that we will present in \cref{sec:obj_val,sec:ch_exactness_polyhedral,sec:convex_hull_exactness}. 
We conclude this section by presenting preliminaries related to the ROG property in \cref{subsec:rog_prelim}.

\item \cref{sec:obj_val} presents sufficient conditions for objective value exactness under the assumption that the set of convex Lagrange multipliers is polyhedral. We define the notion of definite and semidefinite faces of the set of convex Lagrange multipliers and use these definitions to state our first sufficient conditions. \cref{thm:sdp_tightness_main} presents a sufficient condition under which the minimizers of the (projected) SDP relaxation are also the minimizers of the QCQP.
Note that this property is stronger than objective value exactness.
\cref{thm:sdp_tightness_weak} relaxes the sufficient condition of \cref{thm:sdp_tightness_main} when only objective value exactness is required.
We follow these results with a comparison of our sufficient conditions with others proposed in the literature and a number of example applications. Specifically, we will see how to recover objective value exactness results for diagonal QCQPs with sign-definite linear terms and QCQPs with centered constraints and polyhedral convex Lagrange multipliers.

\item \cref{sec:ch_exactness_polyhedral} presents sufficient conditions for convex hull exactness under the assumption that the set of convex Lagrange multipliers is polyhedral. After providing a short proof sketch of this result, we then present a number of example applications. Specifically, we will see how to derive convex hull exactness results for the generalized trust region subproblem, a quadratic problem over a ``Swiss cheese''-like domain, and QCQPs with large amounts of symmetry and polyhedral convex Lagrange multipliers.

\item \cref{sec:convex_hull_exactness} shows how to modify the sufficient conditions presented in \cref{sec:ch_exactness_polyhedral} when the assumption that the set of convex Lagrange multipliers is polyhedral is dropped. This sufficient condition is then used to derive convex hull exactness results for a specific mixed-binary set with complementarity constraints as well as for general quadratic matrix programs.

\item Using the terminology and notation introduced in \cref{subsec:rog_prelim}, we study necessary and/or sufficient conditions under which the intersection of the positive semidefinite cone with a set of (possibly infinitely many) homogeneous LMIs is an ROG cone in \cref{sec:ROG}.

We start by presenting a few consequences of the ROG property in terms of objective value exactness and convex hull exactness in \cref{subsec:QCQP_exactness}. This section also illustrates how results on the ROG property of convex cones can be translated into inhomogeneous SDP exactness results, variants of the S-lemma, and SDP-based convex hull descriptions of quadratically constrained sets. 
We present a few applications of these results, and comment on how ROG-based sufficient conditions for the SDP exactness (objective or convex hull) differ from the previous SDP exactness conditions presented in \cref{sec:ch_exactness_polyhedral,sec:obj_val,sec:convex_hull_exactness}.

In \cref{sec:ROG_LMIs_to_LMEs,sec:ROG_operations,sec:ROG_quadratic_sol}, we build our toolset to study ROG sets. Specifically, in \cref{sec:ROG_LMIs_to_LMEs} we examine their facial structure in terms of how the ROG property behaves when we convert some of the LMIs to LMEs (these results are  particularly useful for analyzing spectrahedral sets defined by finitely many LMIs/LMEs). We present  simple operations preserving the ROG property in \cref{sec:ROG_operations}. We discuss in \cref{sec:ROG_quadratic_sol} the characterization of the ROG property in terms of the existence of nonzero solutions of quadratic systems. 

We discuss a number of sufficient conditions for the ROG property in \cref{sec:sufficient_conditions} and illustrate how to recover some well-known ROG results related to the TRS and the TRS with an additional linear constraint.

In \cref{sec:necessary_conditions}, we give a complete characterization of ROG cones defined by two LMIs, and establish that  there exist simple certificates of the ROG property in this case.  We additionally give a short proof sketch of this characterization in the most interesting case.

In \cref{subsec:quadratic_ratios}, we demonstrate how to apply our ROG toolset to derive objective value exactness results for the problem of minimizing a \emph{ratio} of quadratic functions over a quadratically constrained domain.
\end{enumerate}

We present comparisons of these results with the literature in further detail throughout the document. \subsection{Notation}
For nonnegative integers $m\leq n$, define $[n]\coloneqq\{1,\ldots,n\}$ and $[m,n]\coloneqq \set{m,m+1,\dots,n-1,n}$. 
Let $\R_+$ denote the nonnegative reals.
Let $\S^n$ denote the set of real symmetric $n\by n$ matrices and $\S^n_+$ the cone of positive semidefinite matrices.
We write $A\succeq 0$ (respectively, $A\succ 0$) if $A$ is positive semidefinite (respectively, positive definite). 
Given $A\in\S^n$, let $\range(A)$ and $\ker(A)$ denote the range and kernel of $A$ respectively.
For $M\in \R^{n\by n}$, let $\Sym(M)\coloneqq\tfrac{M + M^\top}{2}\in\S^n$.
For $a\in\R^n$, let $\Diag(a)$ denote the diagonal matrix $A\in\S^n$ with diagonal entries $A_{i,i} = a_i$ for all $i\in[n]$.
Let $0_n,I_n\in\S^n$ denote the $n\by n$ zero matrix and identity matrix respectively; we will simply write $0$ or $I$ when the dimension is clear from context.
We will overload notation and also let $0_n\in\R^n$ denote the zero vector; whether $0$ or $0_n$ is a scalar, vector, or matrix will be clear from context.
For $W$ a subspace of $\R^n$ with dimension $k$, a surjective map $U:\R^k \to W$ and $A\in\S^n$, let $A_W$ denote the restriction of $A$ to $W$, i.e., $A_W = U^*AU$. When the map $U$ is inconsequential, we will omit specifying it.
For $a\in\R^n$, let $\Pi_W a\in W$ denote the orthogonal projection of $a$ onto $W$.
For $u\in W$ and $v\in W^\perp$, let $u \oplus v$ denote their direct sum.
For $A\in \S^n$ and $B\in\S^m$, let $A\oplus B\in\S^{n+m}$ and $A\otimes B\in \S^{nm}$ denote the direct sum and Kronecker product of $A$ and $B$ respectively.
For a subset $\cD$ of some Euclidean space (e.g., $\R^n$ or $\S^n$)
let
$\cD^\circ$,
$\inter(\cD)$,
$\rint(\cD)$,
$\rbd(\cD)$,
$\extr(\cD)$,
$\conv(\cD)$,
$\clconv(\cD)$,
$\cone(\cD)$,
$\clcone(\cD)$,
$\spann(\cD)$,
$\aff(\cD)$,
$\dim(\cD)$,
$\aff\dim(\cD)$ and
$\cD^\perp$
denote the polar, interior, relative interior, relative boundary, extreme points, convex hull, closed convex hull, conic hull, closed conic hull,
lineal hull, affine hull, dimension, affine dimension, and orthogonal complement of $\cD$, respectively.
We will write $\cF\trianglelefteq\cD$ to denote the fact that $\cF$ is a face of $\cD$.
Let $\L^{n+1}\coloneqq \set{z=(x,t)\in\R^{n}\times\R:\, \norm{x}_2\leq t}$ denote the second-order cone (SOC) in $\R^{n+1}$. \section{Preliminaries}\label{sec:Prelim}

\subsection{Problem setup}\label{sec:Setup}
In this tutorial, we will restrict ourselves to considering \textit{epigraphs} of quadratically constrained quadratic programs. See \cite{wang2020geometric} for similar results on general quadratically constrained sets. Let $\cM_I,\cM_E\subseteq\S^{n+1}$ and define $\cM\coloneqq \cM_I\cup\cM_E \cup \set{-M:\, M\in\cM_E}$. Let
\begin{align*}
\cX \coloneqq \set{x\in\R^n:\, \begin{array}
	{l}
	\begin{pmatrix}x\\1\end{pmatrix}^\top M \begin{pmatrix}x\\1\end{pmatrix}\leq 0 ,\,\forall M\in\cM_I\\
	\begin{pmatrix}x\\1\end{pmatrix}^\top M \begin{pmatrix}x\\1\end{pmatrix}= 0 ,\,\forall M\in\cM_E
\end{array}}.
\end{align*}
We will routinely think of a matrix $M\in\S^{n+1}$ as a block matrix with the form $M = \begin{smallpmatrix}
	A & b\\b^\top & c\end{smallpmatrix}\in\S^{n+1}$ for $A\in\S^n$, $b\in\R^n$ and $c\in\R$. In this form, $M\in\S^{n+1}$ defines a quadratic function on $x\in\R^n$ via the map
\begin{align*}
x\mapsto \begin{pmatrix}
	x\\
	1
\end{pmatrix}^\top \begin{pmatrix}
	A & b\\b^\top & c\end{pmatrix} \begin{pmatrix}
	x\\
	1
\end{pmatrix} = x^\top A x + 2b^\top x + c.
\end{align*}
In other words, $\cX\subseteq\R^n$ is a domain defined by a collection of quadratic inequality constraints, $\cM_I$, and a collection of quadratic equality constraints, $\cM_E$.

Given an additional quadratic objective function $\Mobj = \begin{smallpmatrix}
	A_\obj & b_\obj\\b_\obj^\top & c_\obj
\end{smallpmatrix}\in\S^{n+1}$, we will consider the QCQP
\begin{align}\label{eq:opt}
\Opt\coloneqq \inf_{x\in\R^n}\set{\begin{pmatrix}
		x\\1
	\end{pmatrix}^\top \Mobj \begin{pmatrix}
		x\\1
	\end{pmatrix}:\, x\in\cX}.
\end{align}
Let $\cD$ denote the epigraph of this QCQP, i.e.,
\begin{align}\label{eq:D}
\cD &\coloneqq \set{(x,t)\in\R^n\times\R:\, \begin{array}
	{l}
	\begin{pmatrix}
		x\\1
	\end{pmatrix}^\top \Mobj \begin{pmatrix}
		x\\1
	\end{pmatrix}\leq t\\
	x\in\cX
\end{array}}.
\end{align}

In \cref{sec:ch_exactness_polyhedral,sec:obj_val,sec:convex_hull_exactness}, we will additionally assume that $\cM_I$ and $\cM_E$ are both finite. In such a setting, it will be convenient to label the matrices in $\cM_I\cup\cM_E$. Let $m_I$ and $m_E$ denote the number of inequality and equality constraints respectively and set $m \coloneqq m_I + m_E$ to be the total number of constraints.
Let $\cM_I = \set{M_1,\dots,M_{m_I}}$ and $\cM_E = \set{M_{m_I+1}, \dots, M_m}$. For $i\in[m]$, we will write $M_i = \begin{smallpmatrix}
	A_i & b_i\\
	b_i^\top & c_i
\end{smallpmatrix}$ and define
\begin{align*}
q_i(x) = \begin{pmatrix}
	x\\
	1
\end{pmatrix}^\top M_i \begin{pmatrix}
	x\\
	1
\end{pmatrix} = x^\top A_i x + 2b_i^\top x + c_i.
\end{align*}

A large part of our discussion
in \cref{sec:ch_exactness_polyhedral,sec:obj_val,sec:convex_hull_exactness}
will revolve around the Lagrangian dual and aggregation. We thus introduce the following notation.
Let $q:\R^n\to\R^m$ denote the vector-valued function with $q(x)_i = q_i(x)$. Define $A(\gamma)\coloneqq \sum_{i=1}^m \gamma_i A_i$. Similarly define $b(\gamma)$ and $c(\gamma)$. Note that with these definitions
\begin{align*}
\sum_{i=1}^m \gamma_i q_i(x) = \ip{\gamma,q(x)} = x^\top A(\gamma)x+2b(\gamma)^\top x + c(\gamma).
\end{align*}

\subsection{The projected semidefinite programming relaxation}\label{sec:projected_SDP}

A natural convex relaxation of $\cD$ is given by the standard (Shor) semidefinite programming (SDP) relaxation. To simplify our notation, 
given an arbitrary set $\cM\subseteq\S^{n+1}$, we define
\begin{align}\label{eq:S_M}
\cS(\cM)&\coloneqq \set{Z\in\S^{n+1}:\, \begin{array}
	{l}
	\ip{M, Z}\leq 0 ,\,\forall M\in\cM\\
	Z\succeq 0
\end{array}}.
\end{align}
Note that $\cS(\cM)$ is a closed convex cone. We will revisit this set in \cref{subsec:rog_prelim,sec:ROG}.

\begin{definition}
The \emph{projected SDP relaxation of the QCQP} \eqref{eq:opt} is
\begin{align}\label{eq:optsdp}
\Opt_\textup{SDP}&\coloneqq \inf_{x\in\R^n}\set{\ip{M_\obj,Z}:\, \begin{array}
	{l}
	\exists Z = \begin{pmatrix}
		X & x\\ x^\top & 1
	\end{pmatrix}\in\S^{n+1}:\\
	Z\in\cS(\cM)
\end{array}}.
\end{align}

The \emph{projected SDP relaxation of} $\cD$ is
\begin{equation}\label{eq:projected_SDP}
\cD_\textup{SDP}\coloneqq \set{(x,t)\in\R^n\times\R:\, \begin{array}
	{l}
	\exists Z = \begin{pmatrix}
		X & x\\x^\top & 1
	\end{pmatrix}\in\S^{n+1}:\\
	\ip{\Mobj,Z}\leq t\\
	Z\in\cS(\cM)
\end{array}}.\qedhere
\end{equation}
\end{definition}

By taking $X = xx^\top$, it is clear that $\cD\subseteq\cD_\textup{SDP}$. Furthermore, as $\cD_\textup{SDP}$ is the projection of a convex set, it is itself convex. In particular, $\conv(\cD)\subseteq\cD_\textup{SDP}$ and $\Opt\geq \Opt_\textup{SDP}$.

This tutorial presents recent work
\cite{wang2020convex,wang2021tightness,wang2020geometric,argue2020necessary}
towards understanding when equality holds in these relations. Specifically, we will say that a QCQP of the form \eqref{eq:opt} has objective value exactness (resp.\ convex hull exactness) if
$\Opt = \Opt_\textup{SDP}$ (resp.\ $\clconv(\cD) = \cD_\textup{SDP}$).

\subsection{Preliminaries on aggregation}\label{sec:aggregation}
In this subsection, we will present an alternative description of $\cD_\textup{SDP}$ which will highlight the role played by the set of \textit{convex Lagrange multipliers}. We will assume that $\cM_I$ and $\cM_E$ are both finite for the remainder of this subsection.

\begin{definition}
The set of \textit{convex Lagrange multipliers} associated with $\cD$ is
\begin{align*}
\Gamma&\coloneqq \set{(\gobj,\gamma)\in\R\times\R^m:\, \begin{array}
	{l}
	\gobj A_\obj + A(\gamma)\succeq 0\\
	\gobj\geq 0\\
	\gamma_i \geq 0 ,\,\forall i\in[m_I]
\end{array}}.\qedhere
\end{align*}
\end{definition}
Note that $\Gamma$ is a convex cone. This set and its variants are known to be important in studying SDP relaxations of QCQPs; see for example \cite[Chapter 13.4.2]{wolkowicz2012handbook} where a lifted version of this set is used to describe the SDP relaxation of a QCQP (cf.\ \cref{lem:cS_SDP_description}).

We will make the following definiteness assumption which can be interpreted as requiring the dual of \eqref{eq:optsdp} be strictly feasible. This is a standard assumption in the literature.

\begin{assumption}
	\label{as:definiteness}
	There exists $(\gobj^*,\gamma^*)\in\Gamma$ such that $\gobj^* A_\obj + A(\gamma^*)\succ 0$.
\end{assumption}

A version of the following description specifically for the SDP objective value was first recorded by \citet{fujie1997semidefinite}.
\begin{lemma}[{\cite[Lemma 1]{wang2021tightness}}]
\label{lem:cS_SDP_description}
Suppose \cref{as:definiteness} holds. Then,
\begin{align*}
\cD_\textup{SDP} = \set{(x,t)\in\R^n\times \R:\, \gobj(\qobj(x) - t) + \ip{\gamma, q(x)}\leq 0,\,\forall (\gobj,\gamma)\in\Gamma}.
\end{align*}
In particular, $\cD_\textup{SDP} = \set{(x,t)\in\R^n\times \R:\, (\qobj(x)-t, q(x)) \in\Gamma^\circ}$ where $\Gamma^\circ$ is the polar cone of $\Gamma$.
\end{lemma}

In other words, \cref{lem:cS_SDP_description} states that $\cD_\textup{SDP}$ is given by imposing the convex quadratic constraints on $(x,t)\in\R^n\times\R$ that can be obtained via Lagrange aggregation. We illustrate this in greater detail in the following example.

\begin{figure}
 	\centering
 	\hspace{6em}
	\begin{tikzpicture}[scale=1]
		\fill[fill=lightblue] (0,0) -- (0.75,0) -- (1.5,1.5) -- (0,0.75) -- cycle;
		\draw (0.75,0) -- (1.5,1.5) -- (0,0.75);

		\draw[->] (0,0) -- (2,0);
		\draw[->] (0,0) -- (0,2);
		\node[right] (a1') at (2,0) {$\gamma_1$};
		\node[left] (a2') at (0,2) {$\gamma_2$};

		\filldraw (0,0.75) circle (2pt);
		\draw[-, densely dashed] (0,0.75) -- (-1.25,-0.85);
		\filldraw (0,0) circle (2pt);
		\draw[-, densely dashed] (0,0) -- (0.6,-0.85);
		\filldraw (0.75,0) circle (2pt);
		\draw[-, densely dashed] (0.75,0) -- (2.7,-0.85);
		\filldraw (1.5,1.5) circle (2pt);
		\draw[-, densely dashed] (1.5,1.5) -- (6.5,-0.85);
	\end{tikzpicture}
	\\
 	\includegraphics[width=\textwidth / 8, valign=c]{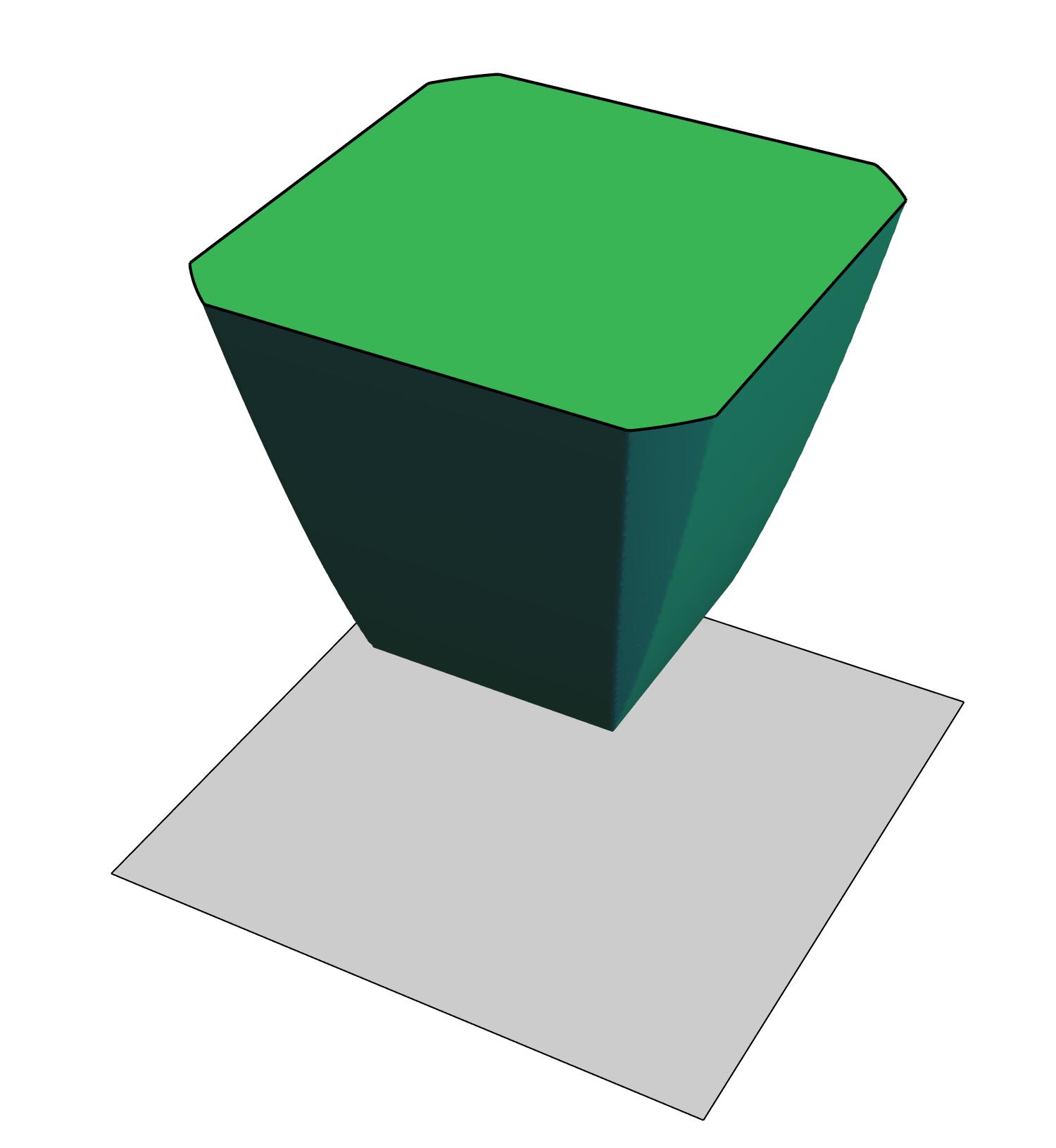}
 	$\,=\,$\includegraphics[width=\textwidth / 8, valign=c]{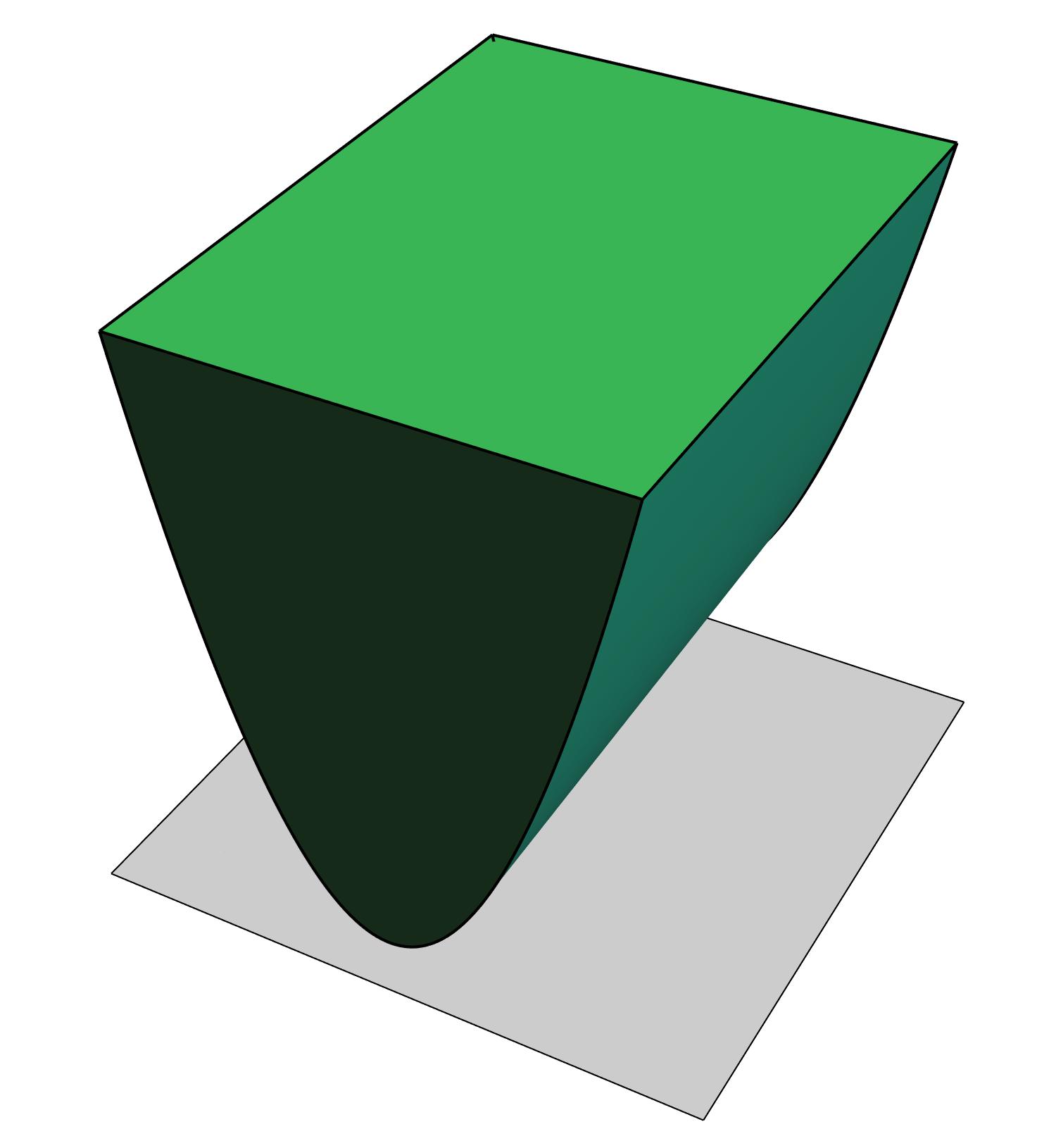}
 	$\cap$\includegraphics[width=\textwidth / 8, valign=c]{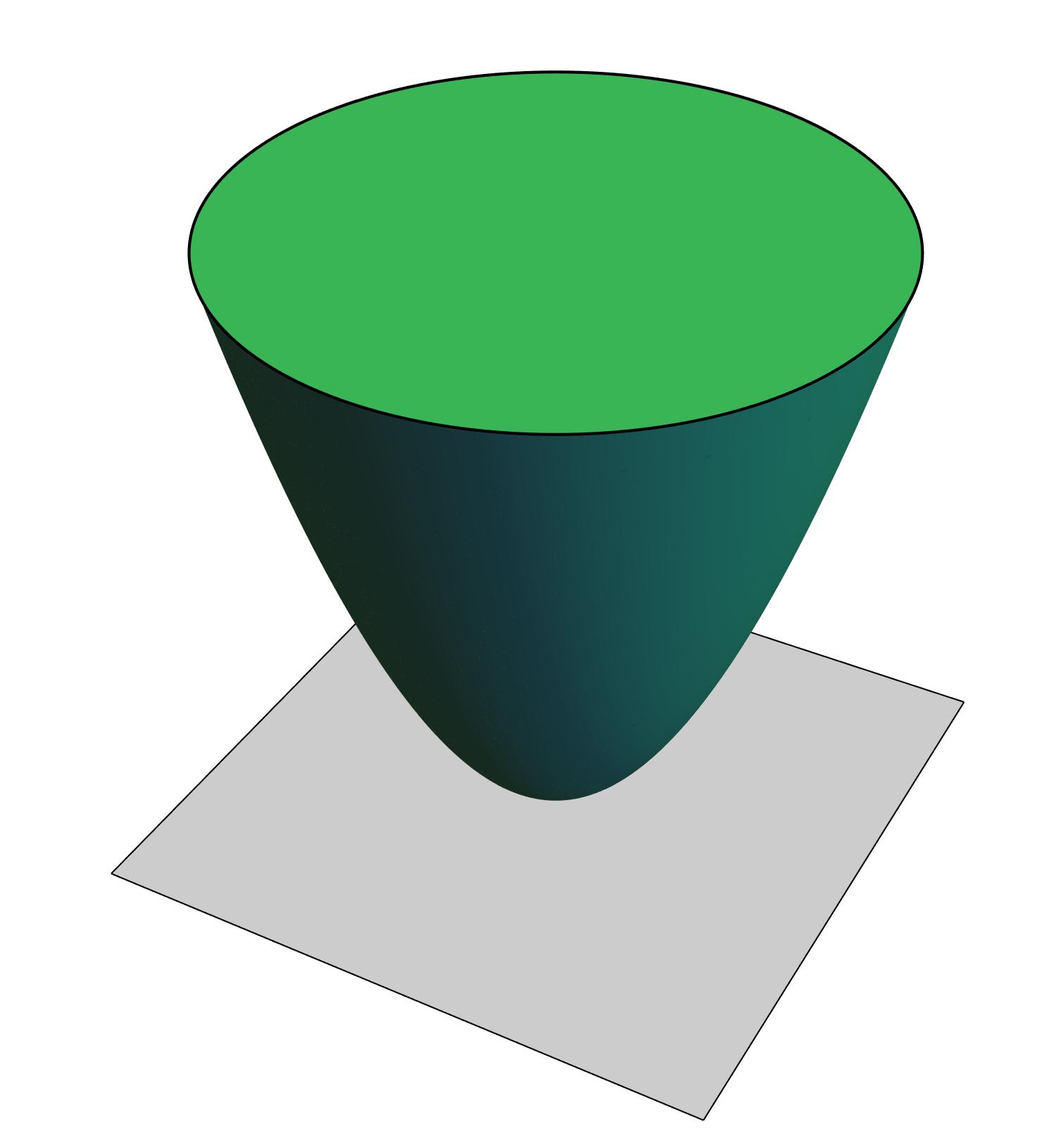}
 	$\cap$\includegraphics[width=\textwidth / 8, valign=c]{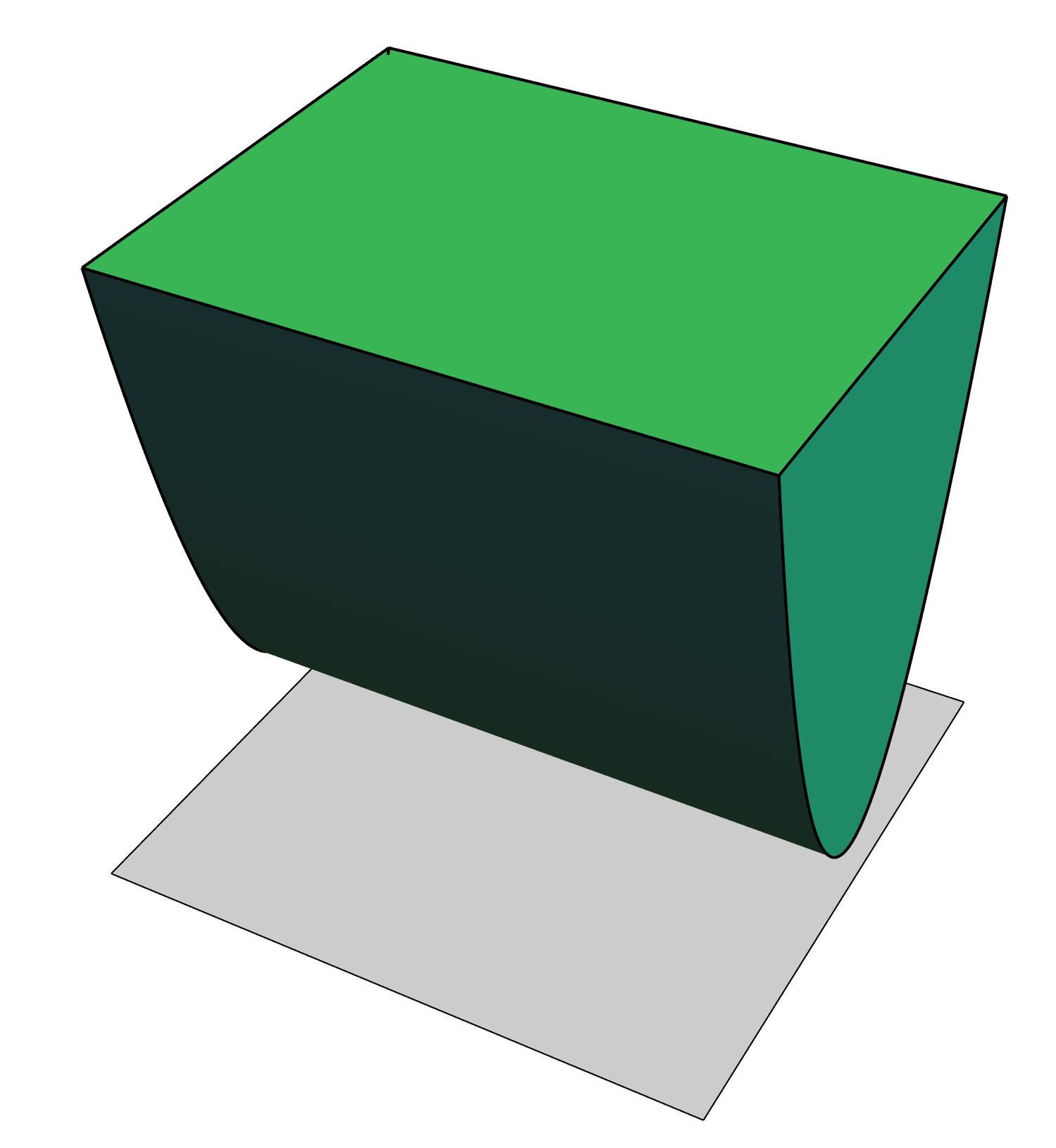}
 	$\cap$\includegraphics[width=\textwidth / 8, valign=c]{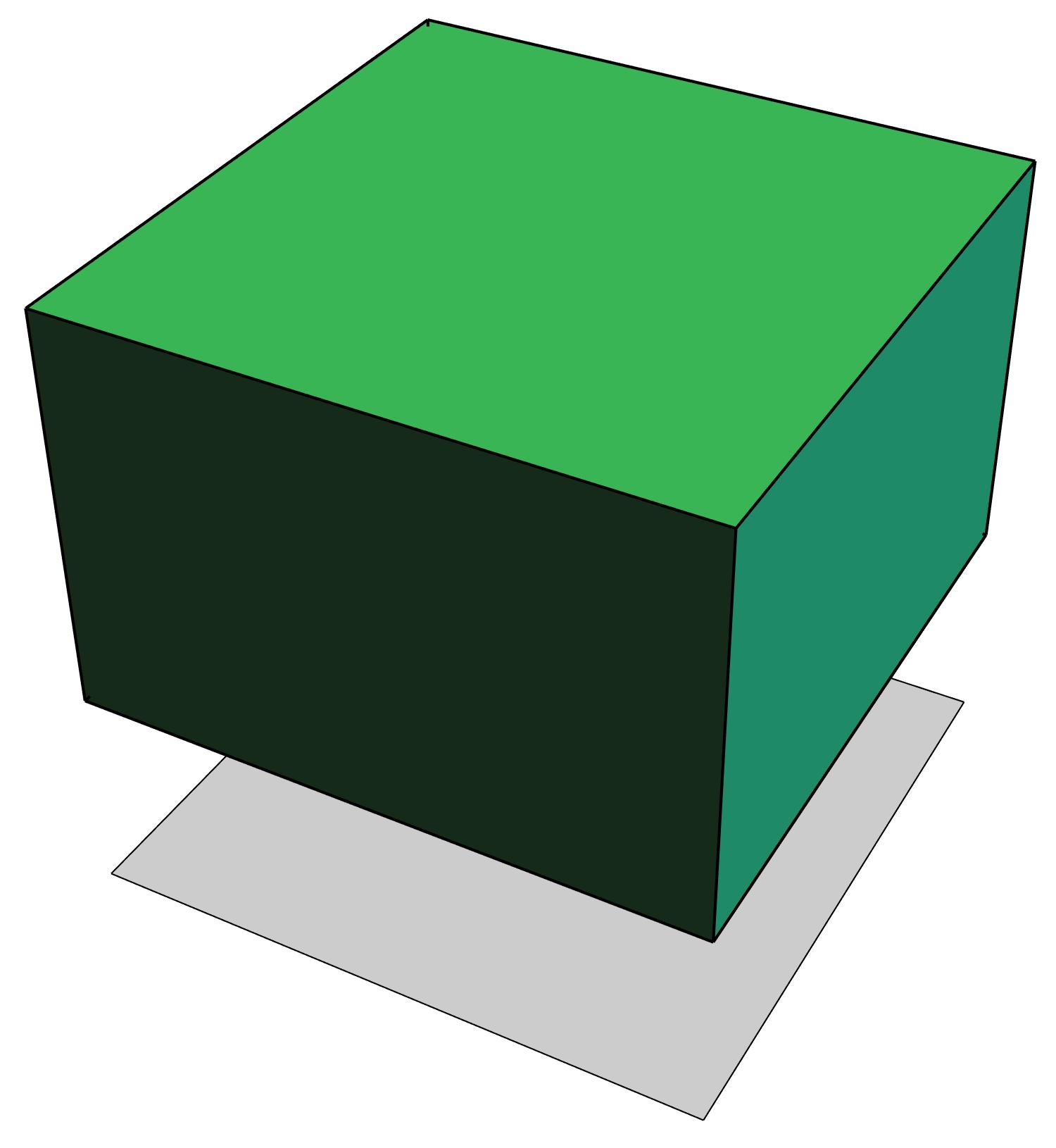}
 	\caption{
 	The blue region (first row) is a plot of the pairs $\gamma\in\R^2$ such that $(1,\gamma)\in\Gamma$ in \cref{ex:explicit_sdp}. \cref{lem:cS_SDP_description} then states that $\cD_\textup{SDP}$ (the leftmost set on the second row) is equal to the intersection of the sets $\set{(x,t)\in\R^2\times \R:\, q_\obj(x) - t + \ip{\gamma,q(x)}\leq 0}$ (the remaining sets on the bottom row) over the extreme points of this blue region.}
 	\label{fig:aggregation}
\end{figure}
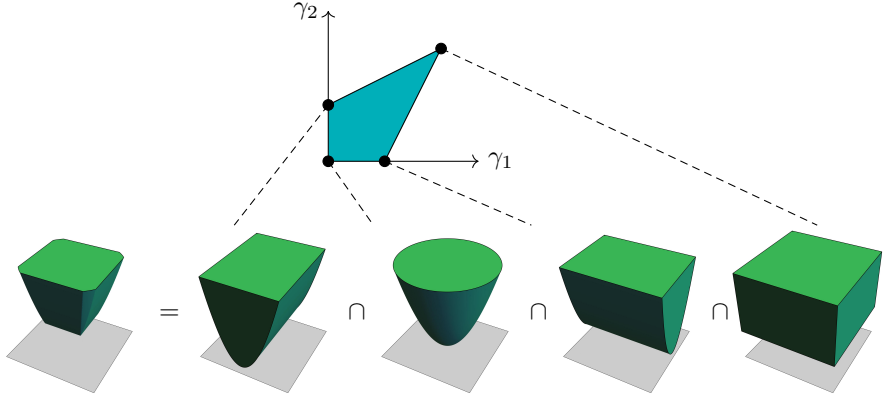

\begin{example}
\label{ex:explicit_sdp}
Consider the following QCQP
\begin{align*}
\inf_{x\in\R^2}\set{x^\top A_\obj x:~ x^\top A_ix + 1 \leq 0,\,\forall i\in[2]\, },
\end{align*}
where
\begin{align*}
A_\obj = \begin{pmatrix}
	1 &\\& 1
\end{pmatrix},\quad
A_1 = \begin{pmatrix}
	-2 &\\& 1
\end{pmatrix},\quad
A_2 = \begin{pmatrix}
	1 &\\& -2
\end{pmatrix}.
\end{align*}
In this case, we may compute $\Gamma$ explicitly:
\begin{align*}
\Gamma &= \set{(\gobj, \gamma)\in\R\times\R^2:\, \begin{array}
	{l}
	\gobj A_\obj + A(\gamma)\succeq 0\\
	\gobj\geq 0\\
	\gamma\geq 0
\end{array}}\\
&= \set{(\gobj, \gamma)\in\R\times\R^2:\, \begin{array}
	{l}
	\gobj -2\gamma_1 + \gamma_2\geq 0\\
	\gobj + \gamma_1 - 2\gamma_2 \geq 0\\
	\gobj\geq 0\\
	\gamma\geq 0
\end{array}}\\
&= \cone\set{
	\begin{pmatrix}1\\	0\\	0\end{pmatrix},
	\begin{pmatrix}1\\	1/2\\	0\end{pmatrix},
	\begin{pmatrix}1\\	0\\	1/2\end{pmatrix},
	\begin{pmatrix}1\\	1\\	1\end{pmatrix}
	}.
\end{align*}
As \cref{as:definiteness} holds (e.g., take $(\gamma_\obj^*,\gamma) = (1,0,0)$), \cref{lem:cS_SDP_description} implies that
\begin{align*}
\cD_\textup{SDP} &= \set{(x,t)\in\R^2\times \R:\, \begin{array}
	{l}
	\gobj(\qobj(x) - t) + \ip{\gamma, q(x)}\leq 0,\,\forall (\gobj,\gamma)\in\Gamma
\end{array}}\\
&= \set{(x,t)\in\R^n\times \R:\, \begin{array}
	{l}
	x_1^2+x_2^2 \leq t\\
	(3/2)x_2^2+1/2 \leq t\\
	(3/2)x_1^2+1/2\leq t\\
	2\leq t
\end{array}}.
\end{align*}
Here, the second line follows as $(\gamma_\obj,\gamma)\mapsto \gamma_\obj(q_\obj(x) - t) + \ip{\gamma,q(x)}$ is linear so that it suffices to impose the constraint in the first line for any set of aggregation weights which generate $\Gamma$. See \cref{fig:aggregation} for a visual depiction of the set $\cD_\textup{SDP}$.

A key takeaway from this example is that when $\Gamma$ is ``simple,'' \cref{lem:cS_SDP_description} gives us a powerful tool to explicitly understand both $\cD_\textup{SDP}$ and $\Opt_\textup{SDP}$.
\end{example}

\begin{remark}
Note that under \cref{as:definiteness}, \cref{lem:cS_SDP_description} implies that $\cD_\textup{SDP}$ is closed.
\end{remark}

\subsection{Rank-one generated subsets of $\S^{n+1}_+$}
\label{subsec:rog_prelim}

In \cref{sec:ROG}, we will work directly with sets of the form  $\cS(\cM)$ defined in \eqref{eq:S_M} and examine the following property of these sets. In contrast to \cref{sec:ch_exactness_polyhedral,sec:convex_hull_exactness,sec:obj_val}, we will not always assume that $\cM$ is finite in this section.
\begin{definition}\label{def:ROG}
A closed convex cone $\cS\subseteq \S^{n+1}_+$ is \emph{rank-one generated} (ROG) if
\begin{align*}
\cS &= \conv(\cS \cap \set{zz^\top:\, z\in\R^{n+1}}).\qedhere
\end{align*}
\end{definition}
\begin{remark}
Note that for a closed convex cone $\cS\subseteq\S^{n+1}_+$, we have $\conv(\cS\cap\set{zz^\top:\, z\in\R^{n+1}}) = \clconv(\cS\cap \set{zz^\top:\, z\in\R^{n+1}})$.
\end{remark}

We will observe in \cref{sec:ROG} that the ROG property 
can be used to derive both objective value and convex hull exactness results.
For example, 
\cref{lem:SDPtightness} states that if $\cS(\cM)\subseteq\S^{n+1}_+$ is ROG, then
objective value exactness holds
\emph{for every} choice of objective function $M_{\obj}$ such that the SDP value is finite.

Recall the following definition.
\begin{definition}
For $Z\in\S^{n+1}$ nonzero, let $\R_+ Z \coloneqq \set{\alpha Z:\, \alpha\geq 0}$ denote the \emph{ray spanned by} $Z$.
We say that $\R_+Z$ is an \emph{extreme ray} of $\cS$ if for any $X,Y\in \cS$ such that $Z = (X+Y)/2$, we have $X,Y\in\R_+Z$.\qedhere
\end{definition}

The following lemma gives an alternate characterization of ROG cones in terms of its extreme rays.
\begin{lemma}[{\cite[Lemma 1]{argue2020necessary}}]
\label{lem:rog_iff_extreme_rank_one}
Let $\cS\subseteq\S^{n+1}_+$ be a closed convex cone. Then, $\cS$ is ROG if and only if for each extreme ray $\R_+ Z$ of $\cS$, we have $\rank(Z)=1$.
\end{lemma}

In contrast to the above characterization, which relates the ROG property of a cone to primal properties (e.g., the rank of its extreme rays), in~\cref{sec:ROG} we will be concerned with understanding the ROG property of a cone $\cS(\cM)$ in terms of its defining inequalities $\cM$. 

\begin{remark}\label{rem:blekherman_comparison}
The ROG property is also relevant in the context of sum-of-squares (SOS) programming. Let
\begin{align*}
V \coloneqq \set{x\in\R^n:\, x^\top A_i x = 0,\,\forall i\in[m]}.
\end{align*}
It is possible to show that ``every quadratic form $A'\in\S^n$ which is nonnegative on $V$ is immediately nonnegative''\footnote{Formally, this is the property that $x^\top A' x\geq 0$ for all $x\in V$ $\implies$
$A'\in \S^n_+ + \spann(\set{A_i})$.} if and only if
\begin{align*}
\set{X\in\S^n_+:\, \ip{A_i, X} = 0 ,\,\forall i\in[m]}
\end{align*}
is ROG.
See \cite[Section 6]{blekherman2017sums} for additional applications and connections of the ROG property in real algebraic geometry and statistics.
\end{remark} \section{Objective value exactness and polyhedral $\Gamma$}
\label{sec:obj_val}

In this section, we take a first look at our framework which is predicated on understanding how the quantities $A(\gamma)$ and $b(\gamma)$ interact on faces of $\Gamma$.
In this section, we will assume that $\Gamma$ is polyhedral and present some conditions under which \textit{objective value exactness}, $\Opt_\textup{SDP} = \Opt$, holds.
Throughout this section, we will assume that $\cM_I$ and $\cM_E$ are both finite.

\begin{assumption}
\label{as:gamma_polyhedral}
The set of convex Lagrange multipliers, $\Gamma$, is polyhedral.
\end{assumption}

Although \cref{as:gamma_polyhedral} is rather restrictive, it is general enough to cover the case
where the set of quadratic forms $\set{A_\obj}\cup \set{A_i:\, i\in[m]}$ is diagonal or simultaneously diagonalizable---a class of QCQPs which have been studied extensively in the literature \cite{benTal2014hidden,locatelli2016exactness}. See also \cite{ramana1997polyhedra} for a characterization of polyhedral spectrahedra and a reduction showing that deciding whether a given spectrahedron is polyhedral is coNP-hard in general. See \cref{fig:gamma_polyhedral_examples} for an illustration of examples and nonexamples of \cref{as:gamma_polyhedral}.

\begin{remark}\label{rem:polyhedral_socp}
We note that under \cref{as:gamma_polyhedral} the set $\cD_\textup{SDP}$ is in fact SOC-representable. Specifically, under this assumption, we may pick a finite subset
$\set{(\gobj^{(j)},\gamma^{(j)})}\subseteq\Gamma$ that generates $\Gamma$, i.e.,
\begin{align*}
\Gamma = \cone\left(\set{(\gobj^{(j)},\gamma^{(j)})}\right).
\end{align*}
Then, as in \cref{ex:explicit_sdp}, $\cD_\textup{SDP}$ is defined by finitely many convex quadratic constraints,
\begin{align*}
\cD_\textup{SDP} = \set{(x,t)\in\R^n\times \R:\, \gobj^{(j)}\ \qobj(x) + \ip{\gamma^{(j)}, q(x)} \leq \gobj^{(j)}\ t ,\,\forall j},
\end{align*}
whence it is SOC-representable.

A number of authors \cite{benTal2014hidden,locatelli2016exactness} have noted that when the set of quadratic forms $\set{A_\obj,A_1,\dots,A_m}$ is simultaneously diagonalizable, that $\cD_\textup{SDP}$ is SOC-representable. It is not hard to show that simultaneous diagonalizability implies \cref{as:gamma_polyhedral} so that we also recover an SOC representability result under simultaneous diagonalizability. In contrast, our SOC representation is in the original space but may involve exponentially many constraints (one for each extreme ray of $\Gamma$), while the SOC representation of $\cD_\textup{SDP}$ given in \cite{benTal2014hidden,locatelli2016exactness} uses $n$ additional variables but only linearly many convex quadratic constraints.
\end{remark}

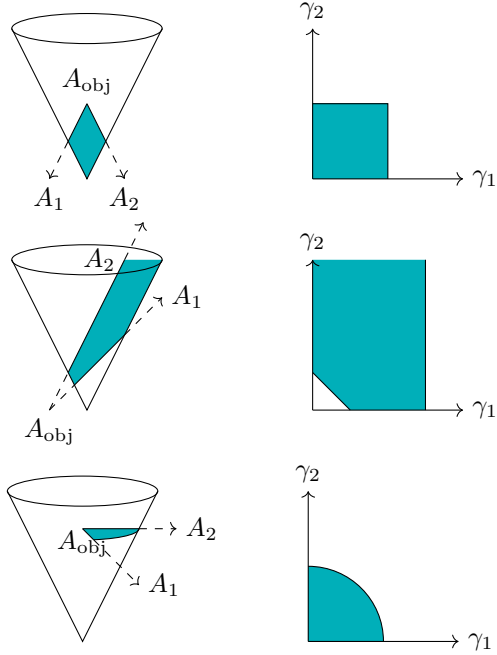
\begin{figure}
 	\centering
	\begin{tikzpicture}
		\fill[fill=lightblue] (0,1) -- (-0.25,0.5) -- (0,0) -- (0.25,0.5) -- cycle;
		\fill[fill=lightblue] (3,0) -- (4,0) -- (4,1) -- (3,1) -- cycle;
		\draw (0,1) -- (-0.25, 0.5);
		\draw[dashed, ->] (-0.25,0.5) -- (-0.5,0);
		\draw (0,1) -- (0.25, 0.5);
		\draw[dashed, ->] (0.25,0.5) -- (0.5,0);
		\node[above] (a0) at (0,1) {$A_\obj$};
		\node[below] (a1) at (-0.5,0) {$A_1$};
		\node[below] (a2) at (0.5,0) {$A_2$};
		\draw (4,0) -- (4,1) -- (3,1);

		\draw (0,2) ellipse (1cm and 0.2 cm);
		\draw (-1,2) -- (0,0) -- (1,2);
		\draw[->] (3,0) -- (5,0);
		\draw[->] (3,0) -- (3,2);
		\node[right] (a1') at (5,0) {$\gamma_1$};
		\node[above] (a2') at (3,2) {$\gamma_2$};
	\end{tikzpicture}

	\begin{tikzpicture}
		\fill[fill=lightblue] (-0.25,0.5) -- (0.5,2) -- (1,2) -- (0.5,1) -- (-1/6,1/3) -- cycle;
		\fill[fill=lightblue] (3.5,0) -- (4.5,0) -- (4.5,2) -- (3,2) -- (3,0.5) -- cycle;
		\draw [dashed](-0.5,0) -- (-0.25, 0.5);
		\draw (-0.25,0.5) -- (0.5,2);
		\draw[dashed,->](0.5,2) -- (0.75, 2.5);
		\draw[dashed] (-0.5,0) -- (-1/6, 1/3);
		\draw(-1/6,1/3) -- (0.5,1);
		\draw[dashed,->] (0.5,1) -- (1,1.5);
		\node[below] (a0) at (-0.5,0) {$A_\obj$};
		\node[left] (a1) at (0.5,2) {$A_2$};
		\node[right] (a2) at (1,1.5) {$A_1$};
		\draw (3.5,0) -- (3,0.5);
		\draw (4.5,0) -- (4.5,2);

		\draw (0,2) ellipse (1cm and 0.2 cm);
		\draw (-1,2) -- (0,0) -- (1,2);
		\draw[->] (3,0) -- (5,0);
		\draw[->] (3,0) -- (3,2);
		\node[right] (a1') at (5,0) {$\gamma_1$};
		\node[above] (a2') at (3,2) {$\gamma_2$};
	\end{tikzpicture}

	\begin{tikzpicture}
		\begin{scope}
			\clip  (0,1.5) ellipse (0.75cm and 0.15 cm);
			\fill[fill = lightblue] (0,1.5) -- (1.5,0) -- (1.5,1.5) -- cycle;
		\end{scope}
		\begin{scope}
			\clip (0,1.5) -- (1.5,0) -- (1.5,1.5) -- cycle;
			\draw  (0,1.5) ellipse (0.75cm and 0.15 cm);
		\end{scope}
		\begin{scope}
			\clip  (3,0) -- (5,0) -- (5,2) -- (3,2) -- cycle;
			\draw[fill = lightblue] (3,0) ellipse (1cm and 1cm);
		\end{scope}
		\draw (0,1.5) -- (0.75,1.5);
		\draw[dashed,->] (0.75,1.5) -- (1.25,1.5);
		\draw (0,1.5) -- (3/sqrt{416} , 1.5 - 3/sqrt{416});
		\draw[dashed,->] (3/sqrt{416} , 1.5 - 3/sqrt{416}) -- (0.75,0.75);
		\node[above] (a0) at (0,1) {$A_\obj$};
		\node[right] (a1) at (0.75,0.75) {$A_1$};
		\node[right] (a2) at (1.25,1.5) {$A_2$};

		\draw (0,2) ellipse (1cm and 0.2 cm);
		\draw (-1,2) -- (0,0) -- (1,2);
		\draw[->] (3,0) -- (5,0);
		\draw[->] (3,0) -- (3,2);
		\node[right] (a1') at (5,0) {$\gamma_1$};
		\node[above] (a2') at (3,2) {$\gamma_2$};
	\end{tikzpicture}
	\caption{In each row above, we illustrate first the set $\set{A_\obj + A(\gamma)\in\S^2:\, \gamma\in\R^2_+}$ on the left and the set $\set{\gamma\in\R^2_+:\, A_\obj + A(\gamma)\in\S^2_+} = \set{\gamma\in\R^2:\, (1,\gamma)\in\Gamma}$ on the right. It is not hard to show that, under \cref{as:definiteness}, $\Gamma$ is polyhedral if and only if the set $\set{\gamma\in\R^2:\, (1,\gamma)\in\Gamma}$ is polyhedral.}
	\label{fig:gamma_polyhedral_examples}
\end{figure}

\begin{definition}
Let $\cF$ be a face of $\Gamma$. We say that $\cF$ is a \emph{definite face} if there exists $(\gamma_\obj,\gamma)\in\cF$ such that $\gamma_\obj A_\obj + A(\gamma)\succ 0$.
Otherwise, we say that $\cF$ is a \emph{semidefinite face}.
\end{definition}
We highlight that based on this definition, a face of $\Gamma$ is \emph{either} a definite face or a semidefinite face. 

\begin{definition}\label{def:cV}
Given a subset $\cF\subseteq\R^{m+1}$, define
\begin{align*}
\cV(\cF) &\coloneqq \set{v\in\R^n:\, v^\top(\gobj A_\obj + A(\gamma))v = 0 ,\,\forall (\gobj,\gamma)\in\cF}.\qedhere
\end{align*}
\end{definition}

The set $\cV(\cF)$ plays an important role in our analysis. The main property we use of this set is the following: Suppose $(\hat x,\hat t)\in\cD_\textup{SDP}$ and $\cF$ is the face of $\Gamma$ exposed by $(q_\obj(\hat x) - \hat t, q(\hat x))$. Then, for $x'\in \cV(\cF)$ and $t'\in\R$, we have that the quadratic constraints associated with $(\gobj,\gamma)\in\cF$, i.e., $\gobj(q_\obj(x) - t) +\ip{\gamma,q(x)}\leq 0$, behave \emph{linearly} when we perturb $(\hat x,\hat t)$ in the direction $(x',2t')$.\footnote{The factor of $2$ here is not important and is only included to unify notation with future sections.} More formally, for any $(\gobj,\gamma)\in\cF$, $x'\in\cV(\cF)$ and $t'\in\R$, the function $\epsilon\mapsto \gobj(q_\obj(\hat x + \epsilon x') - (\hat t + 2\epsilon t')) + \ip{\gamma, q(\hat x + \epsilon x')}$ is linear.

\begin{remark}
Note that if $\cF$ is a subset of $\Gamma$, then
\begin{align*}
\cV(\cF) = \set{v\in\R^n:\, (\gobj A_\obj + A(\gamma))v = 0 ,\,\forall (\gobj,\gamma)\in\cF}
\end{align*}
is the shared zero eigenspace of the matrices corresponding to $\cF$. In particular $\cV(\cF)$ is a linear subspace. It is not hard to show that when $\cF$ is a semidefinite face of $\Gamma$, that $\dim(\cV(\cF))\geq 1$ (see \cite[Lemma 2]{wang2021tightness}). 
\end{remark}

We are now ready to state our first sufficient condition for objective value exactness.
The following result comes from \cite[Theorem 3]{wang2021tightness}.

\begin{theorem}
\label{thm:sdp_tightness_main}
Suppose \cref{as:gamma_polyhedral,as:definiteness} hold. If for every semidefinite face $\cF$ of $\Gamma$ we have 
\begin{align}
\label{eq:sdp_tightness_main_condition}
0\notin \Pi_{\cV(\cF)}\set{b_\obj + b(\gamma):\,
	(1,\gamma)\in\cF}, 
\end{align}
then any optimizer $(x^*,t^*)\in\argmin_{(x,t)\in\cD_{\textup{SDP}}} t$ satisfies $(x^*,t^*)\in\cD$. In particular, $\Opt=\Opt_\textup{SDP}$.
\end{theorem}
We remark that this theorem proves something stronger than objective value exactness. Specifically, \cref{thm:sdp_tightness_main} states that the \emph{only} optimizers of the projected SDP relaxation are the original optimizers.

We give a high-level proof sketch of this statement;
see \cite[Theorem 3]{wang2021tightness} for a complete proof.
We emphasize that the structure of this proof will serve as the foundation for many of the proofs in this tutorial and that we will routinely return to this discussion.

\begin{proof}[Proof sketch of \cref{thm:sdp_tightness_main}]
Let $(\hat x, \hat t)\in\cD_\textup{SDP}$ and let $\cF$ denote the face of $\Gamma$ maximizing the linear function
\begin{align*}
(\gamma_\obj,\gamma)\mapsto \gamma_\obj(\qobj(\hat x) - \hat t) + \ip{\gamma, q(\hat x)}.
\end{align*}
Equivalently, $\cF$ is the face of $\Gamma$ exposed by $(\qobj(\hat x) - \hat t, q(\hat x))$.

If $\cF = \Gamma$, then by the fact that $\Gamma$ is full-dimensional (\cref{as:definiteness}), it must be the case that $\qobj(\hat x) = \hat t$ and $q_i(\hat x) = 0$ for all $i\in[m]$.
More generally, it is possible to show that if $\cF$ is a definite face, then $(\hat x,\hat t)\in\cD$ (see \cite[Lemma 3]{wang2021tightness}).

To prove \cref{thm:sdp_tightness_main}, we will suppose $(\hat x,\hat t)\in\arg\min_{(x,t)\in\cD_\textup{SDP}}t$ corresponds to a semidefinite face of $\Gamma$ and construct a new point $(\hat x+\epsilon x', \hat t+2\epsilon t')\in\cD_\textup{SDP}$ such that $\hat t+\epsilon t'< \hat t$, contradicting the assumption that $(\hat x, \hat t)\in\arg\min_{(x,t)\in\cD_\textup{SDP}}t$.
Specifically, let $x'\in\cV(\cF)$ and $t'\in\R$ correspond to a hyperplane\footnote{If $\set{\gamma:\, (1,\gamma)\in\cF}=\emptyset$, then take $x'=0$ and $t' = -1$.} 
strictly separating the origin from the closed convex set $\Pi_{\cV(\cF)}\set{b_\obj + b(\gamma):(1,\gamma)\in\cF}$ (note that \cref{as:gamma_polyhedral} implies that this set is polyhedral), i.e.,
\begin{align}
\label{eq:str_separation}
\ip{b_\obj + b(\gamma), x'} \leq t' < 0
\end{align}
for all $\gamma\in\R^m$ such that $(1,\gamma)\in\Gamma$.
Then for all $(\gamma_\obj,\gamma)\in\cF$ and $\epsilon>0$, we have
\begin{align*}
\gamma_\obj \left(q(\hat x + \epsilon x') - (\hat t + 2\epsilon t')\right) + \ip{\gamma, q(\hat x + \epsilon x')} &\leq 0
\end{align*}
as $x'\in\cV(\cF)$ and $\ip{b_\obj + b(\gamma), x'} \leq t' < 0$ for all $(1,\gamma)\in\cF$.
On the other hand, for $(\gamma_\obj,\gamma)\in\Gamma\setminus\cF$,
\begin{align}
\label{eq:change_in_aggregated_quadratic}
\gamma_\obj \left(q(\hat x + \epsilon x') - (\hat t + 2\epsilon t')\right) + \ip{\gamma, q(\hat x + \epsilon x')}
\end{align}
is a quadratic function in $\epsilon$ which is negative at $\epsilon= 0$.

Then, using the fact that $\Gamma$ is polyhedral (so that it suffices to ensure that \eqref{eq:change_in_aggregated_quadratic} is nonpositive for only finitely many choices of $(\gamma_\obj,\gamma)$), we deduce that there exists $\epsilon>0$ small enough such that $(\hat x + \epsilon x',\hat t + 2\epsilon t')\in\cD_\textup{SDP}$. This contradicts the assumption that $(\hat x,\hat t)\in\argmin_{(x,t)\in\cD_\textup{SDP}}t$ and concludes the proof sketch.
\end{proof}

\begin{remark}
We highlight some of the key steps in the above proof with  additional intuition and motivation.
Recall that $\Gamma$ is assumed to be polyhedral so that $\cD_\textup{SDP}$ is defined by finitely many convex quadratic constraints. We will for the sake of simplicity assume that $\Gamma = \cone\left(\set{(1,\gamma^{(j)})}\right)$.
The above proof sketch starts with $(\hat x,\hat t)\in\cD_\textup{SDP}$ such that the face $\cF$ of $\Gamma$ exposed by $(q_\obj(\hat x) - \hat t, q(\hat x))$ is semidefinite. This set $\cF$ is the convex hull of the (aggregation weights corresponding to the) convex quadratic constraints that are tight at $(\hat x, \hat t)$.
The remaining finitely many convex quadratic constraints are strictly satisfied at $(\hat x,\hat t)$ so that they remain satisfied under small enough perturbations of $(\hat x,\hat t)$.
In particular, given a perturbation direction $(x',2t')$ we have $(x+\epsilon x', t+2\epsilon t')\in\cD_\textup{SDP}$ for all $\epsilon>0$ small enough if and only if each of the tight constraints continues to hold for all $\epsilon>0$ small enough, i.e.,
\begin{align*}
q(\hat x+\epsilon x') - (\hat t + 2\epsilon t') + \ip{\gamma^{(j)},q(\hat x + \epsilon x')} \leq 0
\end{align*}
for every tight constraint $(1,\gamma^{(j)})\in\cF$ and $\epsilon>0$ small enough.
For a given $x'$, it is possible to show that there exists $t'<0$ satisfying this requirement if and only if each of the tight quadratic constraints is strictly decreasing at $\hat x$ in the direction $x'$.
That is to say, we would like to find a direction $x'$, such that each of the convex quadratic functions $q_\obj(x) + \ip{\gamma^{(j)}, q(x)}$ for $(1,\gamma^{(j)})\in\cF$ has a negative first derivative at $\hat x$ in the direction $x'$.
Expanding, we would like to find $x'$ such that $\ip{(A_\obj + A(\gamma^{(j)}))\hat x + (b_\obj + b(\gamma^{(j)})), x'} <0$ for all $(1,\gamma^{(j)})\in\cF$.
Finally, restricting our search of $x'$ to $\cV(\cF)$, i.e., directions in which all of the tight quadratic constraints are only linear, we then have that such an $(x',2t')$ exists if and only if there exists $x'\in\cV(\cF)$ such that
\begin{align*}
\ip{b_\obj + b(\gamma^{(j)}), x'} &<0,\,\forall (1,\gamma^{(j)})\in\cF.\qedhere
\end{align*}
\end{remark}

A similar proof strategy allows us to relax the conditions of \cref{thm:sdp_tightness_main} if only objective value exactness is required.
\begin{theorem}
\label{thm:sdp_tightness_weak}
Suppose \cref{as:definiteness,as:gamma_polyhedral} hold. If for every semidefinite face $\cF$ of $\Gamma$ we have
\begin{align}
\label{eq:sdp_tightness_weak}
\cone\left(\Pi_{\cV(\cF)}\set{b_\obj + b(\gamma):\, (1,\gamma)\in\cF} \right) \neq \cV(\cF), 
\end{align}
then $\Opt=\Opt_\textup{SDP}$.
\end{theorem}
\begin{proof}[Proof sketch of \cref{thm:sdp_tightness_weak}]
This proof follows a similar structure to the proof of \cref{thm:sdp_tightness_main}. Again, given $(\hat x,\hat t)\in\cD_\textup{SDP}$, we will look at $\cF$, the face of $\Gamma$ exposed by $(q_\obj(\hat x) - \hat t, q(x))$. If $\cF$ is definite, then we may conclude that in fact $(\hat x,\hat t)\in\cD$. On the other hand, if $\cF$ is semidefinite, we will use \eqref{eq:sdp_tightness_weak} to find a direction $(x', t')$ and $\alpha \in\R$ such that $(\hat x+\alpha x', \hat t + \alpha t')\in\cD_\textup{SDP}$ with $\hat t+\alpha t'\leq \hat t$. To complete this proof, one then needs to show that by picking $\alpha$ appropriately and iterating this procedure, we eventually end up with a point in $\cD$.
\end{proof}
\begin{remark}\label{rem:sdp_tightness_weak_rewrite}
Note that the condition in \cref{thm:sdp_tightness_weak} can be rewritten as the existence, for each semidefinite face $\cF$ of $\Gamma$, of a nonzero $x'\in\cV(\cF)$ such that
\begin{align}
\label{eq:sdp_tightness_weak_rewrite}
\ip{b_\obj + b(\gamma),x'} &\leq 0 ,\,\forall (1,\gamma)\in\cF.
\end{align}
This equivalence, i.e., \eqref{eq:sdp_tightness_weak} $\iff$ \eqref{eq:sdp_tightness_weak_rewrite}, follows from the general result for an arbitrary convex subset $C$ of a Euclidean vector space $\E$, that $\cone(C) = \E$ if and only if $\clcone(C) = \E$ if and only if $C^\circ = \set{0}$.

In this form, it is also easy to see that the sufficient condition in \cref{thm:sdp_tightness_main} implies the sufficient condition in \cref{thm:sdp_tightness_weak}, i.e., \eqref{eq:sdp_tightness_main_condition} $\Longrightarrow$ \eqref{eq:sdp_tightness_weak_rewrite}.
Specifically if $\cF$ is a semidefinite face of $\Gamma$ and
$0\notin \Pi_{\cV(\cF)}\set{b_\obj + b(\gamma):\, (1,\gamma)\in\cF}$, then (by the hyperplane separation theorem applied to the convex sets $\Pi_{\cV(\cF)}\set{b_\obj + b(\gamma):\, (1,\gamma)\in\cF}$ and $\set{0}$ in the Euclidean space $\cV(\cF)$) there exists a nonzero $x'\in\cV(\cF)$ such that $\ip{b_\obj + b(\gamma), x'}\leq 0$ for all $(1,\gamma)\in\cF$.
\end{remark}

\begin{remark}\label{rem:comparisons_with_literature}
\citet{burer2019exact} consider the standard SDP relaxation of diagonal QCQPs and present sufficient conditions on the input data that guarantee objective value exactness. Specifically, they show that if for all $j\in[n]$, it holds that
\begin{align}
\label{eq:burer_ye_sufficient}
(1,\gamma)\in\Gamma,\, (A_\obj + A(\gamma))_{j,j} = 0 \implies
(b_\obj + b(\gamma))_j\neq 0,
\end{align}
then objective value exactness holds. It is not hard to see that \eqref{eq:burer_ye_sufficient} implies the assumptions of \cref{thm:sdp_tightness_main} so that \cref{thm:sdp_tightness_main} recovers \cite[Theorem 1]{burer2019exact} as a special case (see \cite[Proposition 5]{wang2021tightness}).

\citet{locatelli2016exactness} considers the standard SDP relaxation of the TRS with additional linear constraints
\begin{align*}
\inf_{x\in\R^n}\set{q_\obj(x):\, \begin{array}
	{l}
	2b_i^\top x + c_i\leq 0,\,\forall i\in[m_I-1]\\
	x^\top x - 1\leq 0
\end{array}}.
\end{align*}
For this setup, \citet{locatelli2016exactness} shows that the SDP relaxation for this problem is exact if for all $\epsilon>0$, there exists $\norm{h_\epsilon}\leq \epsilon$ such that
\begin{align}
\label{eq:locatelli_sufficient}
0\notin \Pi_\cV\set{b_\obj + h_\epsilon + b(\gamma):\, \gamma\in\R^{m}_+},
\end{align}
where $\cV$ is the subspace of $\R^n$ corresponding the minimum eigenvalue (assumed to be negative) of $A_\obj$. \cite[Proposition 4]{wang2021tightness} establishes that \eqref{eq:locatelli_sufficient} implies the assumptions of \cref{thm:sdp_tightness_weak} so that \cref{thm:sdp_tightness_weak} recovers \cite[Theorem 3.1]{locatelli2016exactness} as a special case.
\end{remark}

We next discuss a few applications of these results.
\begin{example}[Convex QCQPs]
As a first example, let us consider the setting where $\cD$ corresponds to a QCQP with a strongly convex objective and convex constraints. Specifically, suppose $A_\obj \succ0$, $A_i\succeq 0$ for all $i\in [m_I]$ and $A_i = 0$ for all $i\in[m_I+1,m]$. Then,
\begin{align*}
\Gamma \coloneqq \set{(\gamma_\obj,\gamma)\in\R\times\R^m:\, \begin{array}
	{l}
	A_\obj + A(\gamma)\succeq 0\\
	\gamma_\obj\geq 0\\
	\gamma_i \geq 0,\,\forall i\in[m_I]
\end{array}} = \R_+\times \R_+^{m_I}\times \R^{m_E}
\end{align*}
is polyhedral. Furthermore, as $A_\obj \succ0$, no semidefinite face $\cF$ of $\Gamma$ can contain a point of the form $(1,\gamma)$.
Then, we deduce by \cref{thm:sdp_tightness_main}, that $\Opt=\Opt_\textup{SDP}$ for convex QCQPs. Here, the requirement that $A_\obj\succ 0$ can be removed by a standard perturbation argument; see \cite[Proposition 4]{burer2019exact}.
\end{example}
\begin{example}[Diagonal QCQPs with sign-definite linear terms]
The following condition is presented in \cite[Corollary 1]{sojoudi2014exactness} (see also \cite{burer2019exact}). Consider the following setup:
Suppose $\cD$ is the epigraph of a diagonal QCQP, i.e., a QCQP where each $A_\obj$ and $A_1,\dots,A_m$ are all diagonal, with only inequality constraints, i.e., $m=m_I$. Furthermore, suppose \cref{as:definiteness} holds and that for all $j\in[n]$, the set of coefficients
\begin{align*}
\set{(b_\obj)_j}\cup \set{(b_i)_j:\, i\in[m]}
\end{align*}
are either all nonnegative or nonpositive.
For all  $j\in[n]$, let $\sigma_j \in \set{\pm 1}$ be such that $(b_\obj)_j\sigma_j \leq 0$ and $(b_i)_j\sigma_j \leq 0$ for all $i\in[m]$.
We can apply \cref{thm:sdp_tightness_weak} in this setting.
Indeed, suppose $\cF$ is a semidefinite face of $\Gamma$.
As $A_\obj$ and $A_1,\dots,A_m$ are diagonal, there exists a nonempty subset of indices $\cJ\subseteq[n]$ such that
\begin{align*}
\cF = \set{(\gamma_\obj,\gamma)\in\Gamma:\, \diag\left(\gamma_\obj A_\obj + A(\gamma)\right)_j = 0 ,\,\forall j\in\cJ}.
\end{align*}
Without loss of generality, we may assume $\cJ = [\ell]$ with $\ell \geq 1$ so that $\sigma_1e_1\in\cV(\cF)$.
Then, for all $(1,\gamma)\in\cF$,
\begin{align*}
\ip{b_\obj + b(\gamma), \sigma_1 e_1} = 
(b_\obj + b(\gamma))_1\sigma_1 = (b_\obj)_1\sigma_1 + \sum_{i=1}^m (b_i)_1 \gamma_i  \sigma_1 \leq 0, 
\end{align*}
where the last inequality follows from $\gamma\in\R^m_+$ (as $(1,\gamma)\in\cF$) and that $(b_\obj)_j\sigma_j \leq 0$ and $(b_i)_j\sigma_j \leq 0$ for all $i\in[m]$.  
Then, we deduce by \cref{thm:sdp_tightness_weak} and \cref{rem:sdp_tightness_weak_rewrite} that $\Opt=\Opt_\textup{SDP}$ for diagonal QCQPs with sign-definite linear terms.
\end{example}
\begin{example}
[QCQPs with centered constraints and polyhedral $\Gamma$]
\label{ex:centered_constraints}
Suppose $b_i = 0$ for all $i\in[m]$ and that \cref{as:definiteness,as:gamma_polyhedral} hold. Then, for any semidefinite face $\cF$ of $\Gamma$, we have
\begin{align*}
\cone\left(\Pi_{\cV(\cF)}\set{b_\obj + b(\gamma):\, (1,\gamma)\in\cF}\right) \subseteq \cone\left(\Pi_{\cV(\cF)}b_\obj\right)
\end{align*}
is a cone generated by a single point and cannot be $\cV(\cF)$ (recall that $\cV(\cF)$ has dimension at least one as $\cF$ is semidefinite). Then, we deduce by \cref{thm:sdp_tightness_weak} that $\Opt=\Opt_\textup{SDP}$ for QCQPs with centered constraints and polyhedral $\Gamma$.
Further specializing this example, we may consider the QCQP
\begin{align*}
\min_{x\in\R^n}\set{-x^\top x:\, \begin{array}
	{l}
	x^\top x \leq 1\\
	x^\top A_i x = 0 ,\,\forall i\in[1,m]
\end{array}},
\end{align*}
which takes the value $-1$ when the quadratic forms $\set{A_i}$ have a nontrivial joint zero and $0$ otherwise. \citet{barvinok1993feasibility} shows that a specialized algebraic algorithm
(not based on the SDP relaxation) can be used to decide the value of this program in polynomial time for any constant $m$. On the other hand, \cref{thm:sdp_tightness_weak} shows that we can decide the value of this program by solving a semidefinite program (for any $m$) whenever \cref{as:gamma_polyhedral} holds. See also \cite[Remark 9]{wang2021tightness}.
\end{example}
\begin{remark}
The conditions presented in \cref{thm:sdp_tightness_main,thm:sdp_tightness_weak} may be strengthened by taking into account the values of $c_i$. For example, it suffices to impose the constraints of \cref{thm:sdp_tightness_main,thm:sdp_tightness_weak} on only the semidefinite faces exposed by vectors of the form $(q_\obj(x) - t, q(x))$ for some $x\in\R^n$ and $t\in\R$. Ideas related to this observation have been studied in further detail in \cite{locatelli2020kkt,wang2020geometric}. See also \cref{sec:convex_hull_exactness}.
\end{remark}

 \section{Convex hull exactness and polyhedral $\Gamma$}
\label{sec:ch_exactness_polyhedral}

In this section, we will continue to assume that $\Gamma$ is polyhedral (\cref{as:gamma_polyhedral}) and present sufficient conditions under this assumption for \emph{convex hull exactness}, $\conv(\cD) = \cD_\textup{SDP}$. As in the previous section, the sufficient conditions stem from understanding how $A(\gamma)$ and $b(\gamma)$ interact on faces of $\Gamma$.
Again, we will assume throughout this section that $\cM_I$ and $\cM_E$ are finite.

The following theorem provides a sufficient condition for convex hull exactness. 
\begin{theorem}[{\cite[Theorem 1]{wang2021tightness}}]
\label{thm:convex_hull_main}
Suppose \cref{as:definiteness,as:gamma_polyhedral} hold. If for every semidefinite face $\cF$ of $\Gamma$ we have
\begin{align}
\label{eq:convex_hull_main}
\aff\left(\Pi_{\cV(\cF)}\set{b_\obj + b(\gamma):\, (1,\gamma)\in\cF}\right)\neq \cV(\cF),
\end{align}
then $\conv(\cD) = \cD_\textup{SDP}$.
\end{theorem}
\begin{proof}[Proof sketch of \cref{thm:convex_hull_main}]
The proof of this statement follows a similar structure to the proofs of \cref{thm:sdp_tightness_main,thm:sdp_tightness_weak}:
Recalling that $\conv(\cD)\subseteq\cD_\textup{SDP}$, it suffices to show that we may write any $(\hat x,\hat t)\in\cD_\textup{SDP}$ as a convex combination of points in $\cD$.
Given $(\hat x,\hat t)\in\cD_\textup{SDP}$, we will examine $\cF$, the face of $\Gamma$ exposed by $(q_\obj(\hat x) - t, q(\hat x))$. If $\cF$ is definite, then we may conclude that in fact $(\hat x,\hat t)\in\cD$. On the other hand, if $\cF$ is semidefinite, we will use \eqref{eq:convex_hull_main} to find a direction $(x', t')$ and $\alpha_1<0<\alpha_2$ such that $(\hat x+\alpha x', \hat t + \alpha t')\in\cD_\textup{SDP}$ for both $\alpha_1$ and $\alpha_2$. To complete this proof, one then needs to show that by picking $\alpha_1,\alpha_2$ appropriately and iterating this procedure for both choices of $\alpha$, we eventually end up with a convex decomposition of $(\hat x,\hat t)$ as points in $\cD$. See \cite[Section 4]{wang2021tightness} for details.
\end{proof}

\begin{remark}
\label{rem:convex_hull_main_rewrite}
We compare the assumptions of \cref{thm:sdp_tightness_weak,thm:convex_hull_main}.
Note that we may rewrite the condition in \cref{thm:convex_hull_main} as the existence, for each semidefinite face $\cF$ of $\Gamma$, of a nonzero $v\in \cV(\cF)$ and $r\in\R$ such that
\begin{align}
\label{eq:convex_hull_main_rewrite}
\ip{b_\obj + b(\gamma), v} = r,\,\forall (1,\gamma)\in\cF.
\end{align}
This equivalence, i.e., \eqref{eq:convex_hull_main} $\iff$ \eqref{eq:convex_hull_main_rewrite}, follows from the observation that an affine subspace is not the entirety of a vector space if and only if it is contained in an affine hyperplane.

In this form, it is clear that the assumptions of \cref{thm:convex_hull_main} imply the assumptions of \cref{thm:sdp_tightness_weak} (see \cref{rem:sdp_tightness_weak_rewrite}). We remark that replacing the inequality in \eqref{eq:sdp_tightness_weak_rewrite} by an equality in \eqref{eq:convex_hull_main_rewrite} is natural when moving from objective value exactness to convex hull exactness. Specifically, in contrast to the proof of \cref{thm:sdp_tightness_weak}, which needs only move in a single direction (i.e., in a direction in which the objective value is nonincreasing), the proof of \cref{thm:convex_hull_main} must move in both directions (i.e., for both $\alpha_1<0$ and $0<\alpha_2$).
\end{remark}

\begin{remark}\label{rem:soc_convex_hull}
Recall that when $\Gamma$ is polyhedral we have that $\cD_\textup{SDP}$ is SOC-representable (see \cref{rem:polyhedral_socp}). In particular, under the assumptions of \cref{thm:convex_hull_main}, we have that $\conv(\cD)$ is also SOC-representable.

A number of results in the literature give second-order cone representations of convex hulls of quadratically constrained sets with a small number of constraints. For example, 
\citet{hoNguyen2017second} show that the epigraph of the TRS is given by the intersection of two convex quadratic regions. \citet{wang2020generalized} extend these results to the GTRS.
\citet{yildiran2009convex} shows that the convex hull of the intersection of two strict quadratic inequalities is given by (open) second-order cone constraints. Follow-up work by \citet{modaresi2017convex} show that it is possible to take the closure under an additional technical assumption.
\citet{burer2017how} examine the convex hull of the intersection of the second-order cone with a nonconvex quadratic constraint. The closed convex hull is shown to be SOC-representable under certain conditions.
More recently, \citet{santana2020convex} showed that the convex hull of the intersection of a polytope and a single nonconvex quadratic constraints is SOC-representable.
\end{remark}

We close this section with example applications of our results.
\begin{example}
[GTRS]
\label{ex:gtrs}
The problem of minimizing a (possibly nonconvex) quadratic function subject to a (possibly nonconvex) quadratic inequality\footnote{Similar statements can also be derived for the GTRS with an equality constraint.} constraint is known as the Generalized Trust Region Subproblem (GTRS). The special case where the quadratic constraint is convex is known as the Trust Region Subproblem (TRS) and is fundamental in the area of nonlinear programming. 
Specifically, trust-region methods (iterative methods for solving nonlinear optimization problems that solve a TRS instance at each iteration)
see strong theoretical guarantees as well as empirical performance \cite{conn2000trust}.

Supposing that \cref{as:definiteness} holds, we may apply \cref{thm:convex_hull_main} to deduce that $\conv(\cD) = \cD_\textup{SDP}$ for the GTRS.
Specifically, in this setting we can write $\Gamma$ as
\begin{align*}
\Gamma = \set{(\gamma_\obj,\gamma_1)\in\R\times\R:\, \begin{array}
	{l}
	\gamma_\obj A_\obj + \gamma_1 A_1 \succeq 0\\
	\gamma_\obj \geq 0\\
	\gamma_1 \geq 0
\end{array}}.
\end{align*}
As $\Gamma$ is a conic subset of $\R^2$, it is immediately polyhedral.
It is not hard to verify that if $\cF$ is a semidefinite face of $\Gamma$, then $\set{\gamma\in\R:\, (1,\gamma)\in\cF}$ is either empty or a single point. We deduce by \cref{thm:convex_hull_main} and \cref{rem:convex_hull_main_rewrite} that $\conv(\cD)=\cD_\textup{SDP}$ for the GTRS. 
This then recovers \cite[Theorem 1]{wang2020generalized} as a special case.
\end{example}

\begin{example}
[Swiss cheese]
\label{ex:swiss_cheese}
Suppose $A_\obj,A_i \in\set{I,0, -I}$ for all $i\in[m]$.
This setting captures, for example, the problem of finding the minimum norm point on some domain defined by ``inside ball'' constraints, ``outside ball'' constraints, and halfspaces. In this setting, for any semidefinite face $\cF$ of $\Gamma$, we have $\cV(\cF) = \R^n$ so that
\begin{align*}
&\aff\dim\left(\Pi_{\cV(\cF)}\set{b_{\obj}+b(\gamma):\, (1,\gamma)\in\cF}\right)
=\aff\dim\set{b_{\obj}+b(\gamma):\, (1,\gamma)\in\cF}\\
&\qquad\leq \aff\dim\set{\gamma:\, (1,\gamma)\in\cF} \leq m - 1.
\end{align*}
Here, the last inequality follows as $\cF$ is a semidefinite face of $\Gamma$ (note that if $\aff\dim\set{\gamma:\, (1,\gamma)\in\cF} = m$, then $\cF$ has affine dimension $m + 1$, which contradicts that $\cF$ is a semidefinite face of $\Gamma$).

We deduce by \cref{thm:convex_hull_main} that if $m_I + m_E = m \leq n$, then $\conv(\cD) = \cD_\textup{SDP}$.

Similar setups have been considered in the literature. For example, \citet{bienstock2014polynomial} devise an enumerative algorithm for minimizing an \emph{arbitrary} quadratic function over a feasible domain defined by \emph{a constant number} of ``inside ball,'' ``outside ball,'' and halfspace constraints. In contrast, our results (via \cref{thm:convex_hull_main}) deal only with objective functions of a particular form, but work with the standard SDP relaxation and do not make any assumption on the number of constraints. See also \cite[Section 2.3]{burer2019exact} for SDP objective value exactness results for an arbitrary quadratic function and an arbitrary number of constraints under a sign-definiteness assumption.

\citet{yang2018quadratic} consider QCQPs with additional ``hollow'' constraints. Formally, they show that if a QCQP with bounded domain $\cX$ satisfies objective value exactness, then so too does the QCQP with domain $\cX \setminus \bigcup_\alpha \inter(E_\alpha)$ where $\set{E_\alpha}$ is a finite set of \emph{non-intersecting} ellipsoids completely contained within $\cX$. Taking $\cX$ to be the unit ball, this result then says that semidefinite programs can correctly minimize an arbitrary quadratic function over the sphere missing a finite number of nonintersecting ellipsoids inside the ball. In contrast, our results (via \cref{thm:convex_hull_main}) deal only with spherical objective functions and constraints (as opposed to general objective functions and ellipsoidal constraints), but do not make any assumption on how the constraints intersect. See also \cite[Remark 10]{wang2021tightness}. The follow-up recent work \cite{joyce2021convex}  
shows that a similar result holds even if $\cX$ is unbounded and the ``hollow'' constraints are not necessarily ellipsoidal.
\end{example}

\begin{example}
[QCQPs with large amounts of symmetry]
\label{ex:qmp}
The following setup is considered in \cite[Section 3]{wang2021tightness}.
Consider a general QCQP of the form \eqref{eq:opt} and let $1\leq k\leq n$ denote the largest positive integer such that each of the quadratic forms $A_\obj,A_1,\dots,A_m$ in the QCQP can be written in the form
\begin{align*}
A_\obj = I_k\otimes \A_\obj,\quad A_i = I_k \otimes \A_i 
\end{align*}
for some $\A_\obj,\A_i\in\S^{n/k}$. This quantity is referred to as the \textit{quadratic eigenvalue multiplicity} of the underlying QCQP and can be thought of as a measure of the ``amount of symmetry'' in the QCQP.
Such structure arises naturally when considering the vectorized reformulation of quadratic matrix programs (QMPs) \cite{beck2007quadratic,beck2012new}. Specifically, a QMP is an optimization problem of the form,
\begin{align*}
\inf_{Y\in\R^{s\times k}}\set{\tr(Y^\top \A_\obj Y) + 2\tr(B_\obj^\top Y) + c_\obj:\, \begin{array}
	{l}
	\tr(Y^\top \A_i Y) + 2\tr(B_i^\top Y) + c_i \leq 0,\,\forall i\in [m_I]\\
	\tr(Y^\top \A_i Y) + 2\tr(B_i^\top Y) + c_i = 0,\,\forall i\in [m_I+1,m]
\end{array}},
\end{align*}
where $\A_i\in\S^{s}$, $B_i\in\R^{s\times k}$, and $c_i\in\R$.
Then,
letting $x\in\R^{sk}$ (resp.\ $b\in\R^{sk}$) denote the vector obtained by stacking the columns of $Y\in\R^{s\times k}$ (resp.\ $B\in\R^{s\times k}$) on top of each other, we have
\begin{align*}
\tr(Y^\top \A Y) + 2\tr(B^\top Y) + c = x^\top (I_k\otimes A) x + 2b^\top x + c.
\end{align*}

The quadratic eigenvalue multiplicity can be viewed as an example of a group symmetry in $\set{A_\obj, A_1,\dots,A_m}$.
Group symmetries have been studied in more generality with the goal of reducing the size of large SDPs \cite{gatermann2004symmetry,deKlerk2007reduction} and have enabled the efficient solution of numerous large-scale problems; see for example \cite{deKlerk2010exploiting}.
Specifically, the \emph{Wedderburn decomposition} of the matrix $\C^*$-algebra generated by $\set{A_\obj, A_1,\dots,A_m}$ plays a prominent role in the analysis of such symmetries (see \cite{deKlerk2011numerical,gijswijt2010matrix} for background on the Wedderburn decomposition and related numerical questions).
From this point of view, one may compare the quadratic eigenvalue multiplicity as defined above with the ``block multiplicity'' of a basic algebra in the Wedderburn decomposition of the corresponding $\C^*$ algebra.
See also \cite[Remark 5]{wang2021tightness} and references therein.

It is not hard to show for any semidefinite face $\cF$ of $\Gamma$, that $\dim(\cV(\cF))\geq k$. Indeed, there exists a nonzero $y\in\R^{n/k}$ such that for all $(\gamma_\obj,\gamma)\in\cF$, we have
$\left(\gamma_\obj \A_\obj + \A(\gamma)\right)y = 0$.
We thus deduce for all $(\gamma_\obj,\gamma)\in\cF$ and $w\in\R^k$ that 
\begin{align*}
\left(\gamma_\obj A_\obj + A(\gamma)\right)(w\otimes y) = 0,
\end{align*}
whence $\dim(\cV(\cF))\geq k$. See also \cite[Lemma 6]{wang2021tightness}.

On the other hand, for any semidefinite face $\cF$ of $\Gamma$, we may upper bound
\begin{align*}
\aff\dim\left(\Pi_{\cV(\cF)}\set{b_\obj + b(\gamma):\, (1,\gamma)\in\cF }\right) \leq \min\left(\abs{\set{i\in[m]:\, b_i \neq 0}},\ m - 1\right).
\end{align*}
We deduce by \eqref{eq:convex_hull_main} that $\conv(\cD) = \cD_\textup{SDP}$ as long as
\begin{align}
\label{eq:ex_qmp_lower_bound}
k \geq \min\left(\abs{\set{i\in[m]:\, b_i \neq 0}} + 1,\ m\right).
\end{align}
Note that this statement immediately subsumes \cref{ex:gtrs,ex:swiss_cheese,ex:centered_constraints}. See \cite[Section 4.3]{wang2021tightness} for constructions showing that our bounds on the value of $k$ guaranteeing both objective value and convex hull exactness are sharp.
\end{example} \section{Removing the polyhedrality assumption}
\label{sec:convex_hull_exactness}
In this section, we show how to extend the results in \cref{sec:obj_val,sec:ch_exactness_polyhedral} to non-polyhedral $\Gamma$.
In contrast to the results thus far, which have been stated in terms of the faces of $\Gamma$, in this section we will focus on the faces of $\Gamma^\circ$, the polar cone of $\Gamma$. The following definitions and results are from \cite{wang2020geometric}.

\begin{definition}
For $(\hat x,\hat t)\in\cD_\textup{SDP}$, let $\cG(\hat x,\hat t)$ denote the minimal face of $\Gamma^\circ$ containing $(q_\obj(\hat x) - \hat t, q(\hat x))$.
\end{definition}

\begin{remark}
Let $\cG$ be a face of $\Gamma^\circ$. We have by definition of $\cV$ (see \cref{def:cV})
\begin{align*}
\cV(\cG^\perp) = \set{v\in\R^n:\, v^\top (\gamma_\obj A_\obj + A(\gamma))v = 0 ,\,\forall (\gamma_\obj,\gamma)\in\cG^\perp}.
\end{align*}
Note that in particular, the definition of $\cV(\cG^\perp)$ contains nonconvex quadratic constraints as $\cG^\perp\setminus\Gamma\neq\emptyset$.
\end{remark}

In the general setting where $\Gamma$ may not be polyhedral, much of our analysis will be done with the object $\cG^\perp$.
This object will replace the face of $\Gamma$ that we used extensively in \cref{sec:obj_val,sec:ch_exactness_polyhedral}. The following example compares $\cG^\perp$ with the face of $\Gamma$ that is natural to consider in this setting. Specifically, we will compare $\cG^\perp$ with $\cG^\triangle$, the conjugate face of $\cG$ in $\Gamma$. Recall that given a face $\cG$ of $\Gamma^\circ$, the conjugate face $\cG^\triangle$ is the face of $\Gamma$ given by $\cG^\triangle\coloneqq \Gamma\cap \cG^\perp$.

\begin{example}
\label{ex:G_perp_vs_G_triangle}
Consider the setting where $\Gamma$ and $\Gamma^\circ$ are both the standard second-order cone in $\R^{n+1}$. Let $\cG$ be a one-dimensional face of $\Gamma^\circ$ and let $\cG^\triangle$ be the face of $\Gamma$ conjugate to $\cG$. Then, $\cG^\triangle$ is a one-dimensional face of $\Gamma$ so that $\spann(\cG^\triangle)$ is a one-dimensional subspace. On the other hand, $\cG^\perp$ is an $n$-dimensional subspace. See \cref{fig:G_perp_vs_G_triangle} for an illustration of the relevant sets for $n=2$.

In general, $\spann(\cG^\triangle)\subseteq \cG^\perp$ where equality may not necessarily hold. On the other hand, it is possible to show that equality holds whenever $\Gamma$ and $\Gamma^\circ$ are polyhedral (see \cite[Theorem 4]{wang2020geometric}).
\end{example}

The following theorem gives a generalization of \cref{thm:convex_hull_main} by establishing a sufficient condition for convex hull exactness without relying on a polyhedrality assumption on $\Gamma$.
\begin{theorem}
\label{thm:convex_hull_general}
Suppose \cref{as:definiteness} holds. If for every $(\hat x,\hat t)\in\cD_\textup{SDP}\setminus\cD$, there exists $x'\in\cV(\cG(\hat x,\hat t)^\perp)$ and $t'\in\R$ such that $(x',t')\neq(0,0)$ and 
\begin{align}
\label{eq:convex_hull_general}
\ip{(A_\obj + A(\gamma))\hat x + (b_\obj + b(\gamma)), x'} = t',\,\forall (1,\gamma)\in\cG(\hat x, \hat t)^\perp ,
\end{align}
then $\conv(\cD) = \cD_\textup{SDP}$.
\end{theorem}

\begin{remark}
It is possible to show (see the proof of \cite[Lemma 2]{wang2020geometric}) that when $\Gamma^\circ$ is facially exposed (as is the case when $\Gamma$ is polyhedral), for every face $\cG$ of $\Gamma^\circ$, we have 
\begin{align*}
\cV(\cG^\perp) = \cV(\cG^\triangle).
\end{align*}
In particular, defining $\cF(\hat x,\hat t) \coloneqq \cG(\hat x,\hat t)^\triangle$, we may rewrite the condition of \cref{thm:convex_hull_general} as the assumption that for every $(\hat x,\hat t)\in\cD_\textup{SDP}\setminus \cD$,
\begin{align*}
\aff\left(\Pi_{\cV(\cF(\hat x,\hat t))}\set{(A_\obj + A(\gamma))\hat x + (b_\obj+b(\gamma)):\, (1,\gamma)\in\cG^\perp}\right) \neq \cV(\cF(\hat x,\hat t)).
\end{align*}
One may compare this condition of \cref{thm:convex_hull_general} with the condition in \cref{thm:convex_hull_main}. We additionally conjecture that $\cV(\cG^\perp) = \cV(\cG^\triangle)$ even without the facially exposed assumption.
\end{remark}

\begin{figure}
	\begin{center}
		\begin{tikzpicture}

			\node[anchor=south] (image) at (5,0) {
				\includegraphics[width=0.25\textwidth]{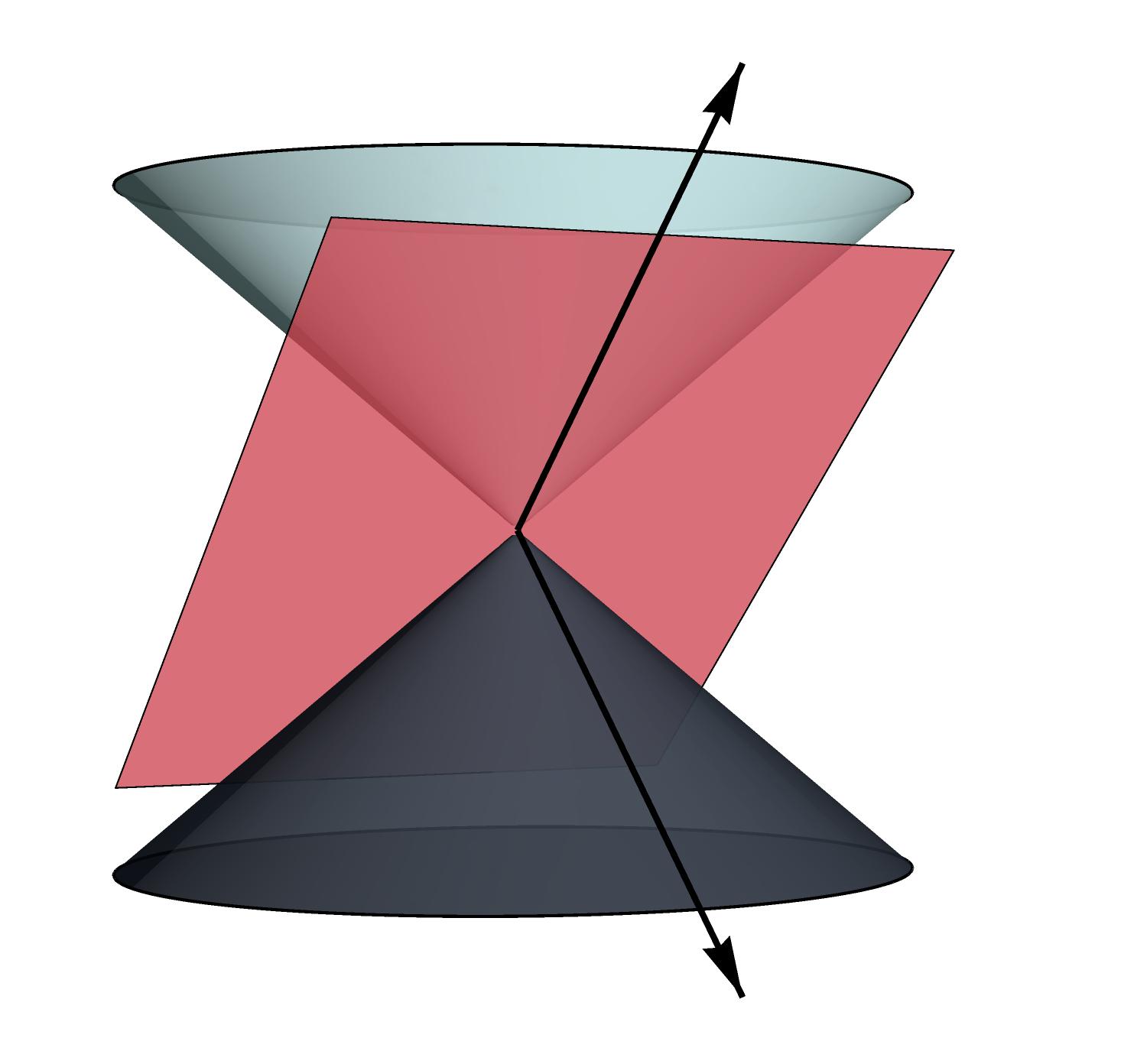}
			};
			\node[anchor=south west] at (5.6,3.7){$\cG^\triangle$};
			\node[anchor=north west] at (5.6,0.38) {$\cG$};
			\node[anchor=south east] at (3.4, 3.33) {$\Gamma$};
			\node[anchor=north east] at (3.4, 0.77) {$\Gamma^\circ$};
			\node[anchor=west] at (6, 2.42) {$\cG^\perp$};
		\end{tikzpicture}
	\end{center}
	\caption{This figure plots $\Gamma$, $\Gamma^\circ$, $\cG$, $\cG^\triangle$, and $\cG^\perp$ from \cref{ex:G_perp_vs_G_triangle}. Note that $\cG^\triangle\subseteq\spann(\cG^\triangle)\subsetneq \cG^\perp$. Specifically, $\cG^\perp$ contains additional directions tangent to $\Gamma$ at $\cG^\triangle$.}
	\label{fig:G_perp_vs_G_triangle}
\end{figure}

\begin{remark}
The main results of \cite{wang2020geometric} show that the sufficient condition presented in \cref{thm:convex_hull_general} is in fact also \emph{necessary} in the setting where $\Gamma^\circ$ is \emph{facially exposed} (see \cite[Theorems 1 and 2]{wang2020geometric}). The condition that $\Gamma^\circ$ is facially exposed holds for example when $\Gamma^\circ$ is a (slice of the) nonnegative cone, second-order cone, or the positive semidefinite cone. See \cite{pataki2013connection} for a longer discussion of this assumption and its connections to the \emph{nice cones}. In general, all nice cones are facially exposed.
\end{remark}

We close this section with two examples that illustrate the use of \cref{thm:convex_hull_general}.
\begin{example}
[A mixed-binary set]
\label{ex:big-M}
Consider the following toy example. Define
\begin{align*}
\cD \coloneqq \set{(x,y,t):\, \begin{array}
	{l}
	x^2  \leq t\\
	x(1-y) = 0\\
	y(1-y) = 0
\end{array}}.
\end{align*}
In words, $y$ is a binary variable and $x$ is constrained to be zero whenever $y$ is ``off.'' Furthermore, $t\geq x^2$.
The convex hull of this set is well known to be given by a perspective reformulation trick~\cite{gunluk2010perspective,ceria1999convex,frangioni2006perspective}. On the other hand, it is also possible to show that $\conv(\cD) = \cD_\textup{SDP}$ using \cref{thm:convex_hull_general} (see \cite[Section 4.3]{wang2020geometric} for details). Specifically, for this example the set $\Gamma$ (whence also $\Gamma^\circ$) is given by a rotated second-order cone. Verifying that the conditions of \cref{thm:convex_hull_general} hold is then a slightly tedious but ultimately straightforward task.
\end{example}

\begin{example}
[QMP]Recall the quadratic matrix programming framework previously considered in \cref{ex:qmp}. Specifically, suppose that for each $i\in[m]$, we can write
\begin{align*}
A_i = I_k \otimes \A_i
\end{align*}
for some $\A_i \in\S^{n/k}$.
Similarly suppose $A_\obj = I_k \otimes \A_\obj$.
Previously, we saw that $\conv(\cD) = \cD_\textup{SDP}$ if $\Gamma$ is polyhedral and $k$ satisfies the lower bound in \eqref{eq:ex_qmp_lower_bound}.
In the case of \emph{arbitrary} $\Gamma$, using \cref{thm:convex_hull_general}, \cite[Section 4.1]{wang2020geometric} shows that that $\conv(\cD) = \cD_\textup{SDP}$ whenever $k\geq m + 2$; see also \cite{wang2021tightness}.
The QMP setting was previously considered by \citet{beck2007quadratic} who proved that $\Opt=\Opt_\textup{SDP}$ whenever $k\geq m$. \cref{table:qmp_exactness} (reproduced from \cite[Remark 14]{wang2021tightness}) gives a summary of the known exactness results for QMPs without polyhedral $\Gamma$ and in particular compares the results of \cite{beck2007quadratic} with the results that we may derive through our framework. See \cite{beck2012new} for further results on exactness in quadratic matrix programming.
\end{example}
\begin{table}
\begin{center}
\begin{tabular}{@{}lll@{}}\toprule
Assumption & Result & Reference\\\midrule
$k \geq m_I + m_E +2$ & $\conv(\cD) = \cD_\textup{SDP}$ &\quad \cite[Theorem 7]{wang2021tightness}\\
$k \geq m_I + m_E +1$ & $\conv(\cD\cap \cH) = \cD_\textup{SDP} \cap \cH$ &\quad \cite[Theorem 8]{wang2021tightness}\\
$k \geq m_I + m_E$ & $\Opt = \Opt_\textup{SDP}$ &\quad \cite[Corollary 4.4]{beck2007quadratic}\\\midrule
\end{tabular}
\caption{A comparison of known exactness results for quadratic matrix programs where $\cH\coloneqq \set{(x,t):\, t = \Opt}$. Note that all of these results assume that \cref{as:definiteness} holds. Here based on the definition of $\cH$, the second row states that \cite[Theorem 8]{wang2021tightness} shows that when $k = m_I + m_E + 1$, the convex hull of the optimizers of the nonconvex QCQP is given by the optimizers of the SDP.}
\label{table:qmp_exactness}
\end{center}
\end{table} \section{Beyond convex hull exactness: Rank-one generated cones}
\label{sec:ROG}

In this section, we move beyond objective value and convex hull exactness and investigate the rank-one generated property of closed convex cones $\cS(\cM)\subseteq\S^{n+1}$.
Recall that we call a closed convex cone
\begin{align*}
\cS(\cM)\coloneqq \set{Z\in\S^{n+1}_+:\, \begin{array}
	{l}
	\ip{M,Z}\leq 0,\,\forall M\in\cM\\
	Z\succeq 0
\end{array}}
\end{align*}
rank-one generated if $\cS(\cM) = \conv(\cS(\cM)\cap \set{zz^\top:\, z\in\R^{n+1}})$.

\begin{example}
\label{ex:rog_first_examples}
The material in this section is best understood with a few examples and nonexamples of ROG cones in mind:
\begin{enumerate}
	\item By the spectral theorem, the PSD cone $\S^n_+$ is ROG. In particular $\cS(\emptyset)$ is ROG.
	\item A well-known result says that $\cS(\cM)$ is ROG whenever $\abs{\cM} = 1$ (see \cite{sturm2003cones} and \cref{lem:M_leq_1}). We illustrate this fact in $\S^2_+$ in \cref{fig:rog_first_examples}.
	\item There exist $\cM$ with $\abs{\cM}=2$ such that $\cS(\cM)$ is not ROG. Specifically, consider $M_1 = \left(\begin{smallmatrix}
		1 &\\
		&-1
	\end{smallmatrix}\right)$ and $M_2 = \left(\begin{smallmatrix}
		0 & 1\\
		1 & 0
	\end{smallmatrix}\right)$. It is clear that $I\in\cS(\cM)$. We claim that $I \notin \conv\left(\cS(\cM)\cap \set{zz^\top:\, z\in\R^2}\right)$. Indeed, supposing otherwise, we deduce from $\ip{M_1,I} = 0$ and $\ip{M_2, I} = 0$ that $I$ is a convex combination of rank-one matrices $zz^\top$ satisfying $0= \ip{M_1, zz^\top} = z_1^2 - z_2^2$ and $0=\ip{M_2,zz^\top} = 2z_1z_2$. This is a contradiction as the only vector $z\in\R^2$ satisfying $z_1^2 - z_2^2 = 2z_1z_2 = 0$ is the zero vector. We illustrate this fact in \cref{fig:rog_first_examples}.\qedhere
\end{enumerate}
\end{example}

\begin{figure}
	\begin{center}
		\begin{tikzpicture}
			\node[anchor=south] (image) at (0,0) {
				\includegraphics[width=0.25\textwidth]{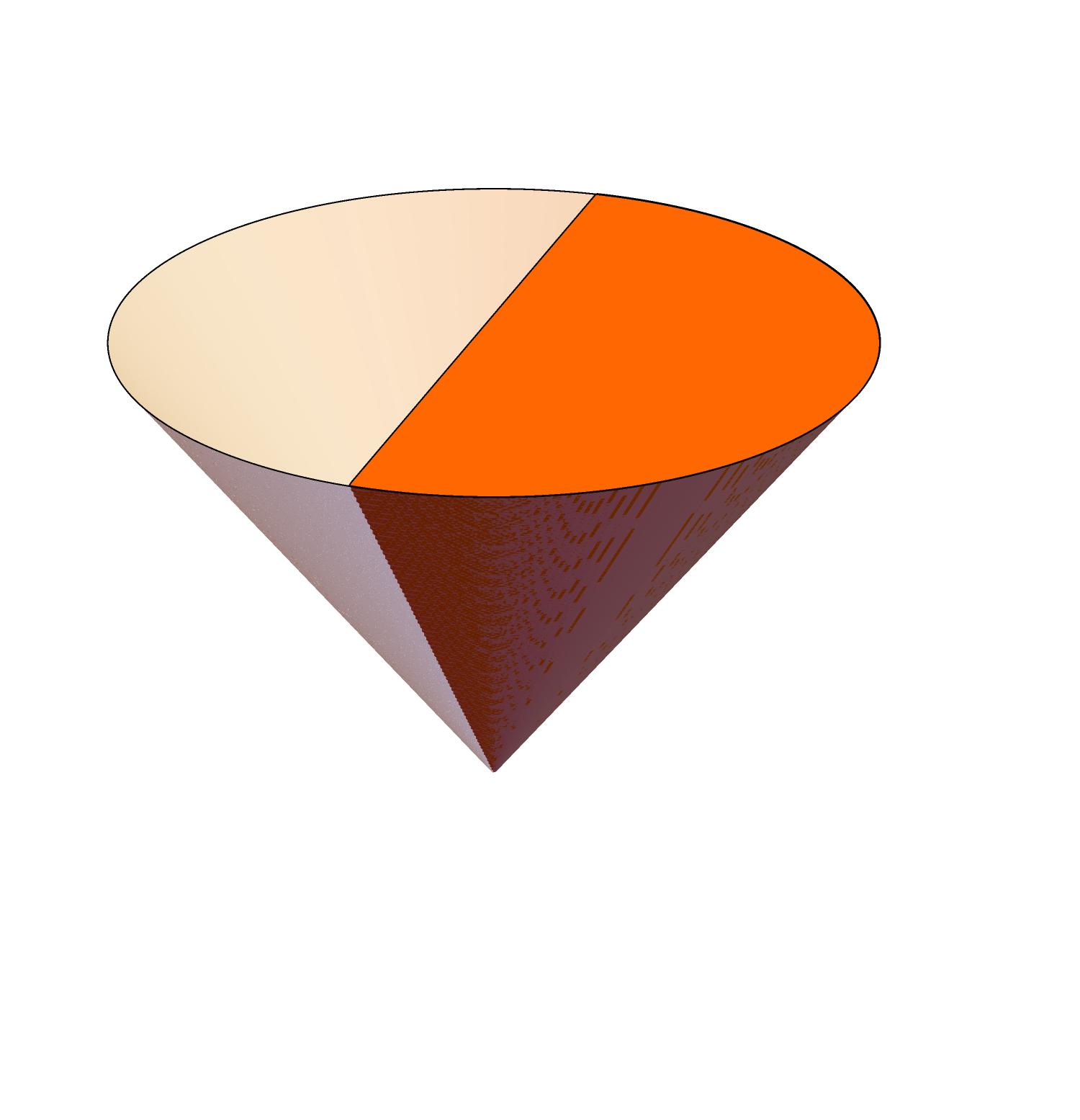}
			};
			\node[anchor=south] (image) at (5,0) {
				\includegraphics[width=0.25\textwidth]{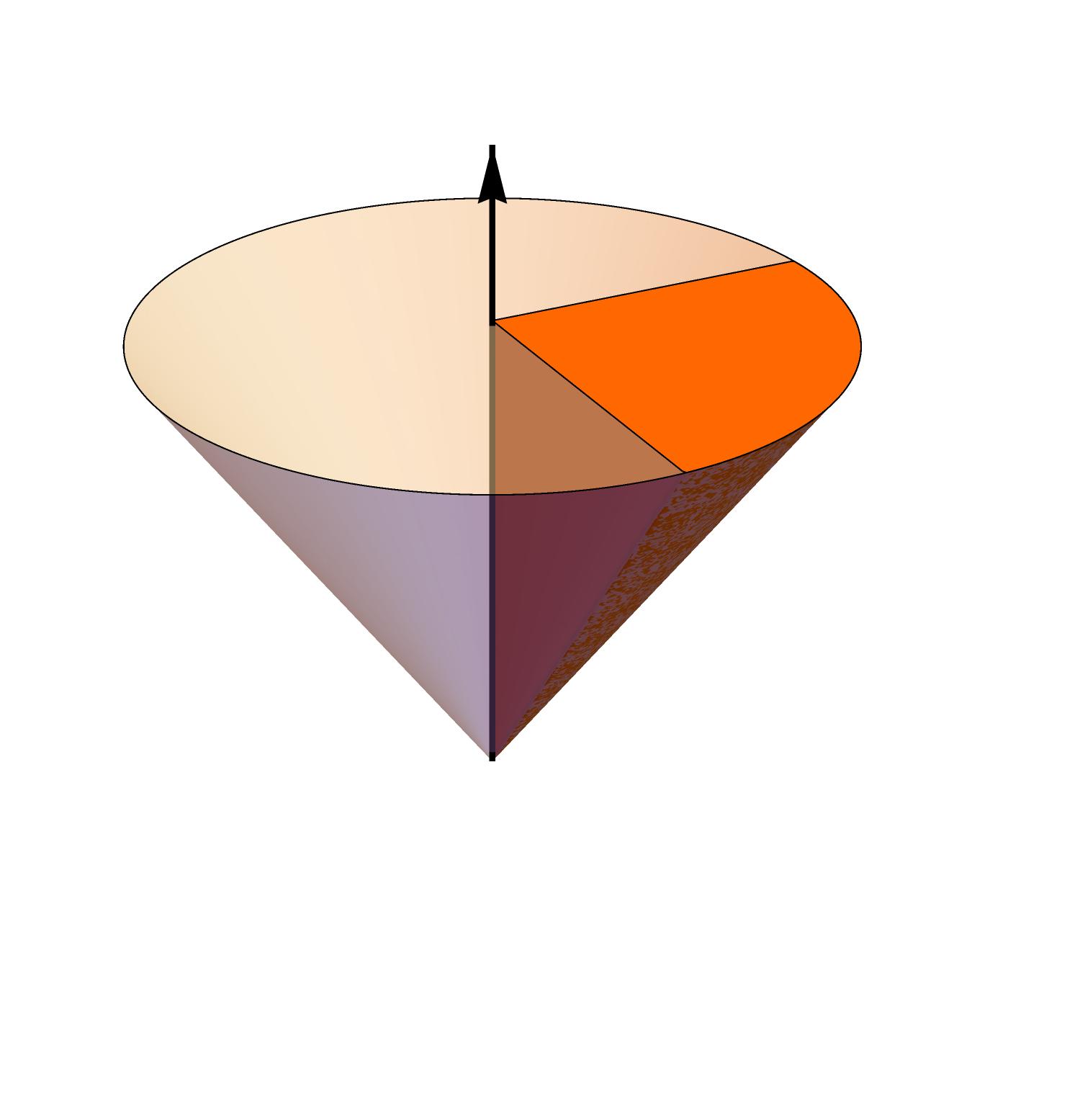}
			};
			\node[anchor=south] at (4.75,3.95) {$I$};
		\end{tikzpicture}
	\end{center}
	\vspace{-3.5em}
  \caption{This figure is adapted from \cite[Figure 3]{argue2020necessary}.
  Every point in the interior of $\S^2_+$ has rank two and every point on the boundary of $\S^2_+$ has rank at most one. The set on the left, $\cS(\set{M_1})$, is ROG as it is equal to the convex hull of its rank-one matrices. The set on the right, $\cS(\cM)$ as defined in the third example of \cref{ex:rog_first_examples}, is not ROG. Specifically, the identity matrix is not in the convex hull of the rank-one matrices belonging to $\cS(\cM)$ in \cref{ex:rog_first_examples}(3).}
  \label{fig:rog_first_examples}
\end{figure}

\cref{subsec:QCQP_exactness} begins by examining connections between the ROG property, variants of the S-lemma, and the exactness conditions studied in the previous sections.
\cref{sec:ROG_LMIs_to_LMEs,sec:ROG_operations,sec:ROG_quadratic_sol} present our toolkit for studying the ROG property.
Using this toolkit, \cref{sec:sufficient_conditions} then derives a number of new sufficient conditions for the ROG property. We complement these sufficient conditions in \cref{sec:necessary_conditions} with our main result (\cref{thm:two_LMI_NS})---a necessary and sufficient condition for the ROG property of spectrahedral cones defined by two LMIs. While we will not formally investigate \cref{thm:two_LMI_NS} until \cref{sec:necessary_conditions}, we will nonetheless mention some of its corollaries before \cref{sec:necessary_conditions}. The interested reader is encouraged to take a look at the statement of \cref{thm:two_LMI_NS} ahead of \cref{rem:ROG_stronger_than_ch_exactness,ex:lifting_rog_to_non_rog}.
We conclude in \cref{subsec:quadratic_ratios} with an application showing how the ROG toolbox can be used to recover SDP exactness results for minimizing \emph{ratios} of quadratic functions over quadratically constrained domains.

\subsection{Connections to objective value exactness and convex hull exactness}
\label{subsec:QCQP_exactness}
The ROG property of a cone $\cS(\cM)$ is intimately related to exactness results
for both homogeneous and inhomogeneous QCQPs and their relaxations.

\subsubsection{Objective value exactness and the ROG property}
We begin with objective value exactness results based on the ROG property. 
To this end, the following lemma (see \cite[Lemma 19]{argue2020necessary}) says that a cone $\cS\subseteq\S^{n}_+$ is ROG if and only if
the SDP relaxation of the corresponding homogeneous QCQP is exact for all choices of objective function.

\begin{lemma}\label{lem:homogeneousSDPtightness}
Let $\cA\subseteq\S^{n}$. Then $\cS(\cA)$ is ROG if and only if for every
$A_\obj\in\S^{n}$,
\begin{align}
\label{eq:homogeneous_sdp_exactness}
\inf_{x\in\R^{n}}\set{\ip{A_\obj,xx^\top}:\, xx^\top\in\cS(\cA)} = \inf_{X\in \S^n}\set{\ip{A_\obj,X}:\,X\in \cS(\cA)}.
\end{align}
\end{lemma}

\cref{lem:homogeneousSDPtightness} is closely related to S-lemma type results. Recall that the S-lemma, which can be traced back to \cite{yakubovich1971s,dines1941mapping}, states the following: Let $A_1\in\S^n$ such that for some $x\in\R^n$, we have $x^\top A_1 x<0$. Then, for all $A_\obj\in\S^n$,
\begin{align*}
x^\top A_1 x \leq 0 \implies x^\top A_\obj x \geq 0
\end{align*}
if and only if there exists $\gamma_1\geq 0$ such that
\begin{align*}
A_\obj + \gamma_1 A_1 \succeq 0.
\end{align*}
In words, the S-lemma gives conditions under which any homogeneous quadratic consequence inequality has a PSD certificate. See also the survey article \cite{polik2007survey} on the S-lemma and its variants.

\cref{lem:homogeneousSDPtightness} allows us to write a variant of the S-lemma in the setting of ROG cones. Specifically, let $\cA\coloneqq \set{A_1,\dots,A_m}\subseteq\S^n$ and suppose there exists $x\in\R^n$ such that $x^\top A_i x< 0$ for all $i\in[m]$. This assumption, a Slater condition, ensures that strong duality holds between the following SDP and its dual:
\begin{align}
\label{eq:rog_slemma_sdp}
\inf_{X\in\S^n}\set{\ip{A_\obj, X}:\, X\in\cS(\cA)}&= \sup_{\gamma\in\R^m_+}\set{0:\,
	A_\obj + \sum_{i=1}^m \gamma_i A_i\succeq 0}.
\end{align}
In particular, \eqref{eq:rog_slemma_sdp} takes the value $0$ if and only if there exists $\gamma\in\R^m_+$ such that $A_\obj + \sum_{i=1}^m A_i \succeq 0$.
On the other hand,
\begin{align*}
\inf_{x\in\R^n}\set{x^\top A_\obj x:\, xx^\top \in\cS(\cA)}
\end{align*}
takes the value $0$ if and only if
$x^\top A_\obj x\geq 0$ for every $x\in\R^n$ such that
$x^\top A_i x\leq 0$ for all $i\in[m]$.

The following statement then follows from the above observations and \cref{lem:homogeneousSDPtightness}.
\begin{corollary}
\label{cor:rog_slemma_hom}
Let $\cA\coloneqq \set{A_1,\dots,A_m}\subseteq\S^n$ and suppose there exists $\bar x\in\R^n$ such that $\bar x^\top A_i \bar x<0$ for all $i\in[m]$. Then, the following are equivalent
\begin{itemize}
	\item For all $A_\obj\in\S^n$,
	\begin{align*}
	x^\top A_i x \leq 0 ,\,\forall i\in[m]\quad\implies\quad
	x^\top A_\obj x \geq 0
	\end{align*}
	if and only if there exists $\gamma\in\R^m_+$ such that 
	\begin{align*}
	A_\obj + \sum_{i=1}^m \gamma_i A_i \succeq 0.
	\end{align*}
	\item $\cS(\cA)$ is ROG.
\end{itemize}
\end{corollary}
In words, under a Slater condition, the S-lemma variant for $\cA$ and an arbitrary $A_\obj$ holds if and only if $\cS(\cA)$ is ROG.

Let us now examine the relation in the case of a general inequality-constrained QCQP and its SDP relaxation.
Note that by introducing a homogenizing variable, a general QCQP \eqref{eq:opt} and its SDP relaxation \eqref{eq:optsdp} can be rewritten as
\begin{align*}
\inf_{z\in\R^{n+1}}\set{\ip{M_\obj, zz^\top}:\, \begin{array}
	{l}
	zz^\top \in\cS(\cM)\\
	\ip{e_{n+1}e_{n+1}^\top, zz^\top} = 1
\end{array}}
&\geq
\inf_{Z\in\S^n}\set{\ip{M_\obj, Z}:\, \begin{array}
	{l}
	Z \in\cS(\cM)\\
	\ip{e_{n+1}e_{n+1}^\top, Z} = 1
\end{array}} .
\end{align*}
In particular, we may always rewrite our QCQPs to contain exactly one inhomogeneous equality constraint.
The following result (see \cite[Lemma 20]{argue2020necessary}) relates the ROG property of a cone to SDP exactness results for its corresponding inhomogeneous QCQPs. 

\begin{lemma}\label{lem:SDPtightness}
Let $\cM\subseteq\S^{n+1}$ and $B\in\S^{n+1}$.
If $\cS(\cM)$ is ROG, then
\begin{align}
\label{eq:SDPtightness}
\inf_{z\in\R^{n+1}}\set{\ip{M_\obj, zz^\top}:\, \begin{array}
	{l}
	zz^\top\in\cS(\cM),\\
\ip{B,zz^\top} = 1
\end{array}}
= \inf_{Z\in\S^{n+1}}\set{\ip{M_\obj, Z}:\, \begin{array}
	{l}
	Z\in\cS(\cM),\\
\ip{B,Z} = 1
\end{array}}
\end{align}
for all $M_\obj\in\S^{n+1}$ for which the optimum SDP objective value is bounded from below. In particular, this equality holds whenever the SDP feasible domain is bounded.
\end{lemma}
By taking $B = e_{n+1}e_{n+1}^\top$ in \cref{lem:SDPtightness}, we have that objective value exactness holds whenever $\cS(\cM)$ is ROG and the SDP optimum value is bounded from below.

The freedom to pick $B\neq e_{n+1}e_{n+1}^\top$ in \cref{lem:SDPtightness} will be useful in \cref{subsec:quadratic_ratios} where we will use it to analyze the problem of minimizing a ratio of quadratic functions over a quadratically constrained domain.

\begin{remark}
We make a few additional observations about \cref{lem:SDPtightness}.
First, \cref{lem:SDPtightness} extends \cite[Lemma 1.2]{hildebrand2016spectrahedral}, which shows that the same statement holds in the case of finitely many linear matrix equalities. The proof in \cite{argue2020necessary} also differs from the proof in \cite{hildebrand2016spectrahedral} as it immediately shows how to construct a QCQP feasible solution achieving the SDP value (or a sequence approaching the SDP value).
Next, we highlight that the reverse implication in \cref{lem:SDPtightness} is not true in general. Specifically, there exist cones $\cS(\cM)$ for which equality holds in \eqref{eq:SDPtightness} for all $M_\obj$ for which the SDP objective value is bounded below that are \emph{not} ROG; see \cite[Example 4]{argue2020necessary}.
Finally, it is possible to show that the boundedness assumption in~\cref{lem:SDPtightness} cannot be dropped even in the case where $B$ is specialized to $B=e_{n+1}e_{n+1}^\top$; see \cite[Example 5]{argue2020necessary}.
\end{remark}

As in \cref{cor:rog_slemma_hom}, we may use \cref{lem:SDPtightness} to derive the following ROG-based variant of the inhomogeneous S-lemma.
\begin{corollary}
Let $\cM\coloneqq \set{M_1,\dots,M_m}\subseteq\S^{n+1}$ and $M_\obj\in\S^{n+1}$ be such that
\begin{itemize}
	\item $\cS(\cM)$ is ROG,
	\item there exists $\bar x\in\R^n$ such that $q_i(\bar x)<0$ for all $i\in[m]$, and
	\item there exists $\bar\gamma\in\R^m_+$ and $\bar\alpha\in\R$ such that $M_\obj + \sum_{i=1}^m \bar\gamma_i M_i - \bar\alpha e_{n+1}e_{n+1}^\top\in\S^{n+1}_+$.
\end{itemize}
Then,
\begin{align*}
q_i(x)\leq 0,\,\forall i\in[m] \quad\implies\quad q_\obj(x) \geq \alpha
\end{align*}
if and only if there exists $\gamma\in\R^m_+$ such that
\begin{align*}
q_\obj(x) + \sum_{i=1}^m \gamma_i q_i(x) \geq \alpha.
\end{align*}
\end{corollary}
In words, under a Slater condition and dual feasibility assumption, the inhomogeneous  S-lemma variant for $\cM$ and an arbitrary $M_\obj$ holds whenever $\cS(\cM)$ is ROG.

\subsubsection{Convex hulls of quadratically constrained sets}

The following proposition (see \cite[Proposition 6]{argue2020necessary}) states that the ROG property of $\cS(\cM)$ guarantees that the closed convex hull of the epigraph of a QCQP with constraints defined by $\cM$ is given by its SDP relaxation.
As before, we will make use of a definiteness assumption.

\begin{proposition}
\label{prop:epi_conv_hull_bounded_quadratic_set}
Let $M_\obj\in\S^{n+1}$ and $\cM\subseteq\S^{n+1}$.
Suppose there exists $M^* = \begin{smallpmatrix}
	A^* & b^*\\b^{*\top} & c^*
\end{smallpmatrix}\in\clcone\left(\set{M_\obj}\cup\cM\right)$ such that $A^*\succ 0$.
If $\cS(\cM)$ is ROG, then 
$\clconv(\cD)= \cD_\textup{SDP}$.
\end{proposition}

Note that when $\cM$ is finite, the assumption that $M^*$ exists is equivalent to \cref{as:definiteness}.
We can further relax this assumption by applying a perturbation argument to arrive at the following result (see \cite[Corollary 6]{argue2020necessary}). 
\begin{corollary}
  \label{cor:epi_conv_hull_bounded_quadratic_set_psd}
Let $M_\obj\in\S^{n+1}$ and $\cM\subseteq\S^{n+1}$.
Suppose there exists $M^* = \begin{smallpmatrix}
	A^* & b^*\\b^{*\top} & c^*
\end{smallpmatrix}\in\clcone\left(\set{M_\obj}\cup\cM\right)$ such that $A^*\succeq 0$.
If $\cS(\cM)$ is ROG, then 
$\clconv(\cD)=\cl( \cD_\textup{SDP})$.
\end{corollary}

The following example recovers \cite[Theorem 4]{wang2020generalized} (see also \cref{ex:gtrs}) as an immediate application of \cref{cor:epi_conv_hull_bounded_quadratic_set_psd}.
\begin{example}[GTRS]
Recall that the GTRS asks us to minimize a quadratic objective function $q_\obj(x)$ subject to a single quadratic inequality constraint $q_1(x) \leq 0$.
Here, we assume that $q_\obj,q_1$ are such that there exists $\gamma^*\geq 0$ such that $A_\obj + \gamma^* A_1\succeq 0$. 
It is well known (see for example \cite{sturm2003cones}) that the cone $\cS(\set{M_1})$ is ROG. We then deduce by \cref{cor:epi_conv_hull_bounded_quadratic_set_psd} that
\begin{align*}
\clconv\set{(x,t)\in\R^n\times\R:\, \begin{array}
	{l}
	q_\obj(x) \leq t\\
	q_1(x) \leq 0
\end{array}} &=
\cl\set{(x,t)\in\R^n\times \R:\, \begin{array}
	{l}
	\exists X\succeq xx^\top\,:\\
	\ip{A_\obj, X} + 2\ip{b_\obj, x} + c_\obj \leq t\\
	\ip{A_1, X} + 2\ip{b_1, x} + c_1 \leq 0
\end{array}}.\qedhere
\end{align*}
\end{example}

We close this subsection by demonstrating that the ROG property is (unsurprisingly) strictly stronger than convex hull exactness.
\begin{remark}
\label{rem:ROG_stronger_than_ch_exactness}
Consider the following QCQP
\begin{align*}
\inf_{x\in\R^2}\set{\norm{x}^2:\, \begin{array}
	{l}
	x^\top \left(\begin{smallmatrix}
	-1 &\\
	& 1
\end{smallmatrix}\right)x - 1\leq 0\\
x^\top \left(\begin{smallmatrix}
	2 &\\
	& -1
\end{smallmatrix}\right)x - 1\leq 0
\end{array}}.
\end{align*}
The corresponding set $\cM$ for this example is $\cM = \set{\Diag(-1,1,-1),\Diag(2,-1,-1)}$.
We will soon see in \cref{thm:two_LMI_NS} that $\cS(\cM)$ is not ROG so that \cref{prop:epi_conv_hull_bounded_quadratic_set} cannot be applied to this example.
On the other hand, the set of convex Lagrange multipliers $\Gamma$ for this QCQP is polyhedral so that we may apply the analysis of \cref{ex:qmp} to deduce that convex hull exactness holds.
\qedhere
\end{remark}

\subsection{Relating LMIs to LMEs}\label{sub:relating_SM_and_TM}\label{sec:ROG_LMIs_to_LMEs}

In this section, we will present a series of lemmas relating the ROG property of $\cS(\cM)$ to that of
\begin{align*}
\cT(\cM') \coloneqq \set{Z\in\S^{n+1}_+:\, \ip{M,Z} = 0,\,\forall M\in\cM'}
\end{align*}
for $\cM'\subseteq \cM$. Note that in contrast to $\cS(\cM)$, the definition of $\cT(\cM')$ uses linear matrix \emph{equalities} (LMEs).
These results will be particularly useful for analyzing spectrahedral sets defined by finitely many LMIs/LMEs. 

\begin{remark}\label{rem:TM_finite_SM_infinite}
For any $\cM\subseteq\S^{n+1}$, we have
$\cS(\cM) = \cS(\clcone(\cM))$ and $\cT(\cM) = \cT(\spann(\cM))$. Consequently, we may without loss of generality assume that $\cM$ is finite when analyzing sets of the form $\cT(\cM)$---simply replace $\cM$ with a finite basis of $\spann(\cM)$. On the other hand, $\clcone(\cM)$  is not necessarily finitely generated.\qedhere
\end{remark}

The following result from \cite[Lemma 3]{argue2020necessary} relates the facial structure of $\cS(\cM)$ to the ROG property. This result is analogous to the statement that a polytope is integral if and only if each of its faces is integral.
\begin{lemma}
\label{lem:rog_iff_faces_rog}
For any set $\cM\subseteq \S^{n+1}$, the following are equivalent:
\begin{enumerate}
	\item $\cS(\cM)$ is ROG.
	\item Every face of $\cS(\cM)$ is ROG.
	\item $\cS(\cM)\cap \cT(\cM')$ is ROG for every $\cM'\subseteq\cM$.
\end{enumerate}
\end{lemma}

We have the following immediate corollary of \cref{lem:rog_iff_faces_rog};  see \cite[Corollary 1]{argue2020necessary}.
\begin{corollary}
\label{cor:SM_rog_then_TM_rog}
For any set $\cM\subseteq \S^{n+1}$, if $\cS(\cM)$ is ROG then $\cT(\cM)$ is ROG.
\end{corollary}

We can strengthen \cref{lem:rog_iff_faces_rog} in a few ways (see \cite[Lemmas 5 and 6]{argue2020necessary}).

\begin{lemma}
\label{lem:rog_iff_nontrivial_faces_rog}
Let $\cM\subseteq \S^{n+1}$ be compact. Then, $\cS(\cM)$ is ROG if and only if $\cS(\cM)\cap\cT(\cM')$ is ROG for every $\emptyset\neq \cM'\subseteq\cM$.
\end{lemma}

\begin{lemma}
\label{lem:TMprime_rog_then_SM_rog}
Let $\cM\subseteq \S^{n+1}$ be finite.
If $\cT(\cM')$ is ROG for every $\cM'\subseteq \cM$, then $\cS(\cM)$ is ROG.
\end{lemma}

We next note \cite[Lemma 7]{argue2020necessary} that the ROG property of $\cT(\cM)$ is equivalent to
the ROG property of $\cT(\overline\cM)$ where $\overline\cM$ is the restriction
of $\cM$ onto the joint range of the matrices $M\in\cM$.
\begin{lemma}
\label{lem:TM_rog_iff_overline_TM_rog}
Let $W \coloneqq \spann\left(\bigcup_{M\in\cM} \range(M)\right)$. For $M\in\cM$, let $\overline M= M_W$ denote the restriction of $M$ to $W$. 
Let $\overline\cM = \set{\overline M:\, M\in\cM}$.
Then, $\cT(\cM)$ is ROG if and only if $\cT(\overline\cM)$ is ROG.
\end{lemma}

\begin{figure}
	\centering
		\begin{tikzcd}
		\boxed{\cM \text{ is finite and } \forall \cM'\subseteq \cM,\, \cT(\cM')\text{ ROG}} \arrow[r,Rightarrow]
		& \boxed{\cS(\cM)\text{ ROG}} \arrow[r,Rightarrow]
		& \boxed{\cT(\cM)\text{ ROG}}
		\end{tikzcd}
	\caption{A summary of \cref{lem:TMprime_rog_then_SM_rog,cor:SM_rog_then_TM_rog}}
	\label{fig:relating_S_and_T}
\end{figure}

Considering the definitions of $\cS(\cM)$ and $\cT(\cM)$, one may ask whether it is possible to ``lift'' the inequalities in $\cS(\cM)$ to equalities while preserving the ROG property.
More formally, given a finite set $\cM =\set{M_1,\dots,M_m}$ such that $\cS(\cM)$ is ROG, is the set $\cT(\overline\cM)\subseteq\S^{n+m}$ ROG? Here,
\begin{align*}
\overline\cM = \set{\overline M_1,\dots,\overline M_m}
\quad\text{and}\quad
\overline M_i = \begin{pmatrix}
	M_i &\\& -e_ie_i^\top
\end{pmatrix}.
\end{align*}
If this were possible, then there would be no need to study the ROG property of sets of the form $\cS(\cM)$.
The following example (\cite[Example 3]{argue2020necessary}) shows that this is not possible in general.
\begin{example}
\label{ex:lifting_rog_to_non_rog}
Consider the set
\begin{align*}
\cS \coloneqq \set{Z\in\S^3_+:\, \begin{array}
	{l}
	Z_{1,2} = 0\\
	Z_{1,3} \leq 0
\end{array}}.
\end{align*}
We will see soon (see \cref{thm:two_LMI_NS}) that this set is ROG.
We can replace the LMIs defining $\cS$ with LMEs in a lifted space as follows: Let $\Pi:\S^4\to\S^3$ denote the projection of a $4\by 4$ matrix onto its top-left $3\by 3$ principal submatrix. Then,
\begin{align*}
\cS = \Pi\left(\set{Z\in\S^4:\, \begin{array}
	{l}
	Z_{1,2} = 0\\
	Z_{1,3} + Z_{4,4} = 0
\end{array}}\right) = \Pi\left(\cT(\set{\overline M_1,\overline M_2})\right),
\end{align*}
where
\begin{align*}
\overline M_1 \coloneqq \begin{pmatrix}
	0 & 1/2 & 0 & 0\\
	1/2 & 0 & 0 & 0\\
	0 & 0 & 0 & 0\\
	0 & 0 & 0 & 0
\end{pmatrix}
\qquad\text{and}\qquad
\overline M_2 \coloneqq \begin{pmatrix}
	0 & 0 & 1/2 & 0\\
	0 & 0 & 0 & 0\\
	1/2 & 0 & 0 & 0\\
	0 & 0 & 0 & 1
\end{pmatrix}.
\end{align*}

Define $\overline\cM\coloneqq\set{\overline M_1,\overline M_2}$. \cref{thm:two_LMI_NS} implies that $\cT(\overline\cM)$ is not ROG. We conclude that the obvious lifting of LMIs into LMEs can take ROG sets $\cS(\cM)$ to non-ROG sets $\cT(\overline\cM)$.\qedhere
\end{example}

\subsection{Simple operations preserving the ROG property}\label{sec:ROG_operations}

The following results are useful in reasoning about extreme rays of $\cS(\cM)$, and thus inferring about operations preserving the ROG property.

It is easy to observe that an arbitrary intersection of ROG cones is ROG if and only if no new extreme rays are introduced; see \cite[Lemma 10]{argue2020necessary}.
\begin{lemma}
\label{lem:no_new_extreme}
Let $\cM\subseteq\S^{n+1}$ be a union $\cM = \bigcup_{\alpha\in A} \cM_\alpha$. Suppose that $\cS(\cM_\alpha)$ is ROG for every $\alpha\in A$.
Then, $\cS(\cM)$ is ROG if and only if
\begin{align*}
\extr(\cS(\cM)) \subseteq \bigcap_{\alpha\in A} \extr(\cS(\cM_\alpha)).
\end{align*}
\end{lemma}

\begin{remark}
The statement of \cref{lem:no_new_extreme} is slightly stronger than one might expect given the situation for integral polyhedra.
Specifically, the intersection of a (finite) collection of integral polyhedra is polyhedral if no new extreme points are introduced while the converse does not necessarily hold.
The more accurate analogy for the ROG property in \cref{lem:no_new_extreme} is to the integrality property of polyhedral sets defined with \emph{pure binary} variables.
\end{remark}

\cite[Lemma 9]{argue2020necessary} establishes that we can refine the above result further when $\cM$ can be partitioned into ``non-interacting'' sets of constraints.
This lemma, which we state below, allows us to construct ROG cones out of simpler ROG cones.
\begin{lemma}\label{lem:non-interactingSufficientCond}
Let $\cM \subset \S^{n+1}$ be a finite union of compact sets $\cM = \bigcup_{i=1}^k \cM_i$.
Further, suppose that for all nonzero $Z\in\S^{n+1}_+$ and $i\in[k]$, if $\ip{M_i,Z} = 0$ for some $M_i\in\cM_i$, then $\ip{M, Z}<0$ for all $M\in\cM\setminus \cM_i$.
Then, $\cS(\cM)$ is ROG if and only if $\cS(\cM_i)$ is ROG for all $i\in[k]$.
\end{lemma}

\subsection{The ROG property and solutions of quadratic systems}\label{sec:ROG_quadratic_sol}

The ROG property of a set is naturally connected to the existence of nonzero solutions of underlying systems of quadratic inequality or equality constraints. 
We examine this connection next. 
To this end, for 
$\cM\subseteq \S^{n+1}$ and $Z\in\cS(\cM)$, define
\begin{align*}
\cE(Z,\cM) &\coloneqq \set{z\in\R^{n+1}:\,  \ip{M,Z} \leq z^\top Mz \leq 0,\,\forall M\in\cM}.
\end{align*}

Based on this, we have the following characterization of the ROG property \cite[Lemma 11]{argue2020necessary}.
\begin{lemma}
\label{lem:SM_rog_iff_nonzero_envelope}
$\cS(\cM)$ is ROG if and only if for every nonzero $Z\in\cS(\cM)$ we have $\range(Z)\cap \cE(Z,\cM)\neq \set{0}$.
\end{lemma}

Note that by definition, $\cT(\cM)=\cS(\cM')$ where $\cM'\coloneqq \set{M:\, M\in\cM} \bigcup  \set{-M:\, M\in\cM}$. Therefore, in the case of $\cT(\cM)$, the set $\cE(Z,\cM)$ in \cref{lem:SM_rog_iff_nonzero_envelope} is replaced with a simpler set corresponding to solutions to a homogeneous system of quadratic equations.
In particular, for
$\cM\subseteq \S^{n+1}$, define
\begin{align*}
\cN(\cM) &\coloneqq \set{z\in\R^{n+1}:\, z^\top M z = 0 ,\,\forall M\in\cM}.
\end{align*}

\begin{remark}\label{rem:NM_EXM_rel}
It is easy to see that for every $\cM\subseteq \S^{n+1}$ and $Z \in \cS(\cM)$, we have $\cN(\cM) \subseteq \cE(Z,\cM)$.
\qedhere
\end{remark}
Therefore, we arrive at the following corollary; see \cite[Corollary 2]{argue2020necessary}.
\begin{corollary}
\label{cor:TM_rog_iff_nonzero_null}
$\cT(\cM)$ is ROG if and only if for every nonzero $Z\in \cT(\cM)$ we have $\range(Z)\cap \cN(\cM)\neq\set{0}$.
\end{corollary}

Note that for any rank-one matrix $Z=zz^\top$ in $\cS(\cM)$, we always have $z\in\range(Z)\cap \cE(Z,\cM)$. 
Therefore, when applying \cref{lem:SM_rog_iff_nonzero_envelope}, it suffices to check the right hand side only for matrices $Z$ with rank at least two. The same is true for \cref{cor:TM_rog_iff_nonzero_null}.

\begin{figure}
	\centering
	\begin{tikzcd}
		\boxed{\cS(\cM)\text{ ROG}} \arrow[r,Leftrightarrow]\arrow[d,Rightarrow]
		&\boxed{\forall Z\in\cS(\cM)\setminus\set{0},\,\range(Z)\cap\cE(Z,\cM)\neq \set{0}} \arrow[d,Rightarrow]\\
		\boxed{\cT(\cM)\text{ ROG}} \arrow[r,Leftrightarrow]
		&\boxed{\forall Z\in\cT(\cM)\setminus\set{0},\,\range(Z)\cap\cN(\cM)\neq \set{0}}
		\end{tikzcd}
	\caption{A summary of \cref{lem:SM_rog_iff_nonzero_envelope} and \cref{cor:TM_rog_iff_nonzero_null}.}
\end{figure}

Our tool set allows us to recover a number of known results from the literature.
The following result regarding spectrahedral cones defined by a single LMI is from \citet{sturm2003cones}. 
\begin{lemma}
\label{lem:M_leq_1}
Consider any $M\in\S^{n+1}$, and let $\cM = \set{M}$. Then $\cS(\cM)$ is ROG.
\end{lemma}
Using our tool set we see that \cref{lem:M_leq_1} follows from \cref{lem:rog_iff_nontrivial_faces_rog,cor:TM_rog_iff_nonzero_null}.
In the case of ROG sets defined by two LMIs, \cref{lem:TMprime_rog_then_SM_rog,lem:M_leq_1,cor:SM_rog_then_TM_rog} lead to the following characterization (see \cite[Corollary 3]{argue2020necessary}).
\begin{corollary}
\label{cor:M2_SM_rog_iff_TM_rog}
Suppose $\abs{\cM} = 2$, then $\cS(\cM)$ is ROG if and only if $\cT(\cM)$ is ROG.
\end{corollary}

In \cref{sec:necessary_conditions} we will present necessary and sufficient conditions for a $\cT(\set{M_1,M_2})$ to be ROG; see \cref{thm:two_LMI_NS}. Combined with \cref{cor:M2_SM_rog_iff_TM_rog}, this then gives a complete characterization of ROG sets of the form $\cS(\set{M_1,M_2})$.

\subsection{Sufficient conditions}\label{sec:sufficient_conditions}

In this subsection, we present a number of sufficient conditions for the ROG property.
We begin with a result related to the S-lemma~\cite{fradkov1979s-procedure} and a convexity theorem due to~\citet{dines1941mapping}.
\begin{lemma}\label{lem:psd_sum_rog}
Let $\cM = \set{M_1,M_2}$ and suppose there exists $(\alpha_1,\alpha_2)\neq (0,0)$ such that $\alpha_1M_1+\alpha_2M_2 \in \S^{n+1}_+$. Then, $\cS(\cM)$ is ROG.
\end{lemma}
The proof of this lemma is straightforward and follows from \cref{cor:M2_SM_rog_iff_TM_rog}.

\begin{remark}
  \label{rem:geometric_interpretation_i}
  The condition in \cref{lem:psd_sum_rog} has a simple geometric interpretation.
  Specifically, this condition guarantees that the two LMEs defining
  $\cT(\set{M_1,M_2})$ only interact with each other on a single (possibly
  trivial) face of the positive semidefinite cone. Furthermore, on this face,
  the two LMEs impose the same (possibly trivial) constraint. See \cref{fig:nonintersecting_ROG}.\qedhere
\end{remark}

The following result from \cite[Proposition 1]{argue2020necessary}
generalizes \cref{lem:psd_sum_rog}.

\begin{proposition}
\label{prop:gen_psd_sum_rog_SM}
Let $\cM\subseteq\S^{n+1}$ be finite. Suppose that for every distinct pair $M, M'\in\cM$, there exists $(\alpha,\beta)\neq (0,0)$ such that $\alpha M + \beta M'$ is positive semidefinite. Then, $\cS(\cM)$ is ROG.
\end{proposition}
Intuitively, the conditions in this proposition have a similar geometric
interpretation to the conditions in \cref{lem:psd_sum_rog} (see
\cref{rem:geometric_interpretation_i}). 
Specifically, it is possible to show that for every $\cM'\subseteq\cM$, the set $\cT(\cM')$ is contained in some face of the positive semidefinite cone on which every constraint defining $\cT(\cM')$ imposes the same constraint. See \cref{fig:nonintersecting_ROG}.

\begin{figure}
	\begin{center}
		\begin{tikzpicture}
			\node[anchor=south] (image) at (-1,0) {
				\includegraphics[width=0.25\textwidth]{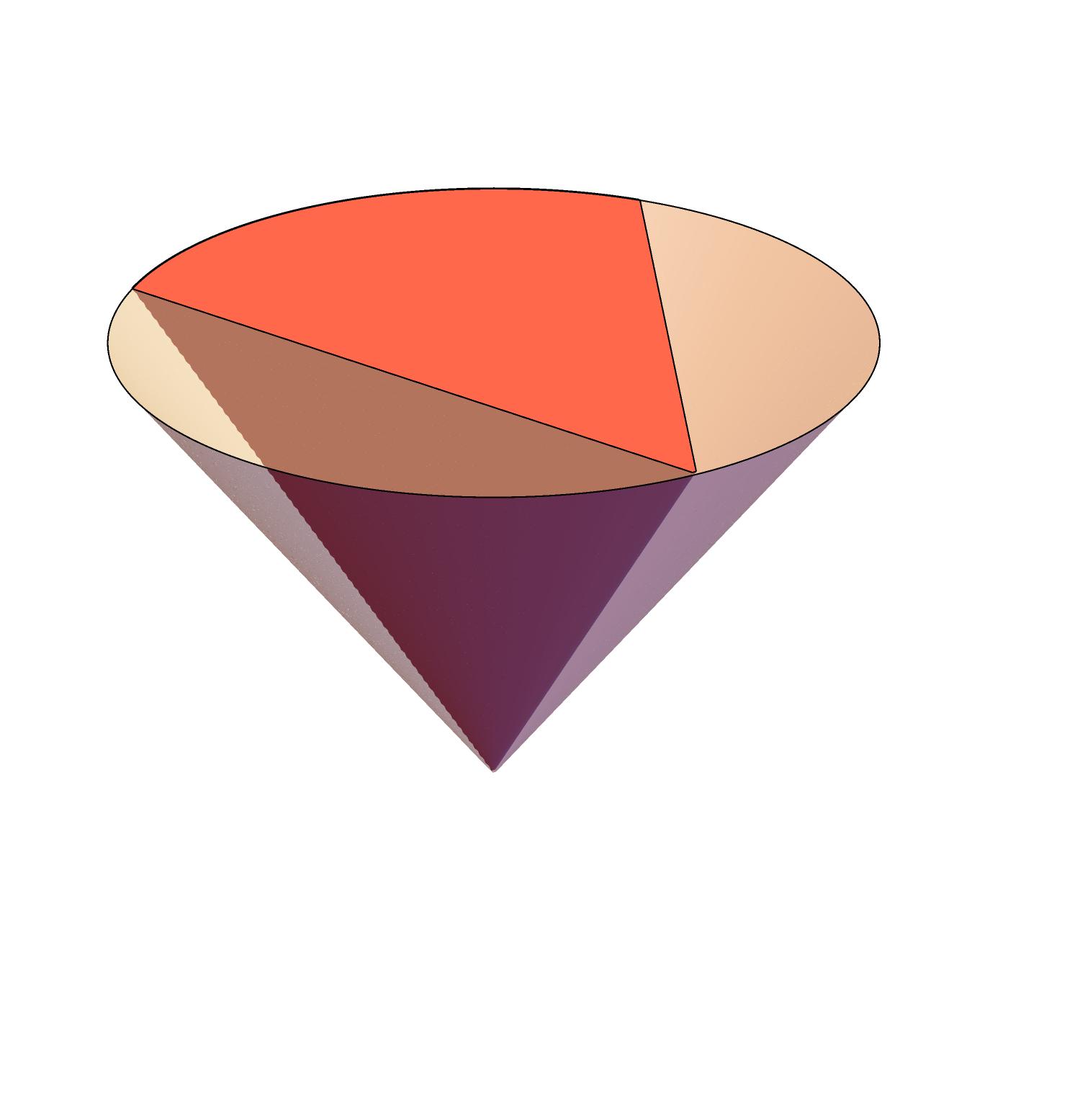}
			};

			\node[anchor=south] (image) at (5,0) {
				\includegraphics[width=0.25\textwidth]{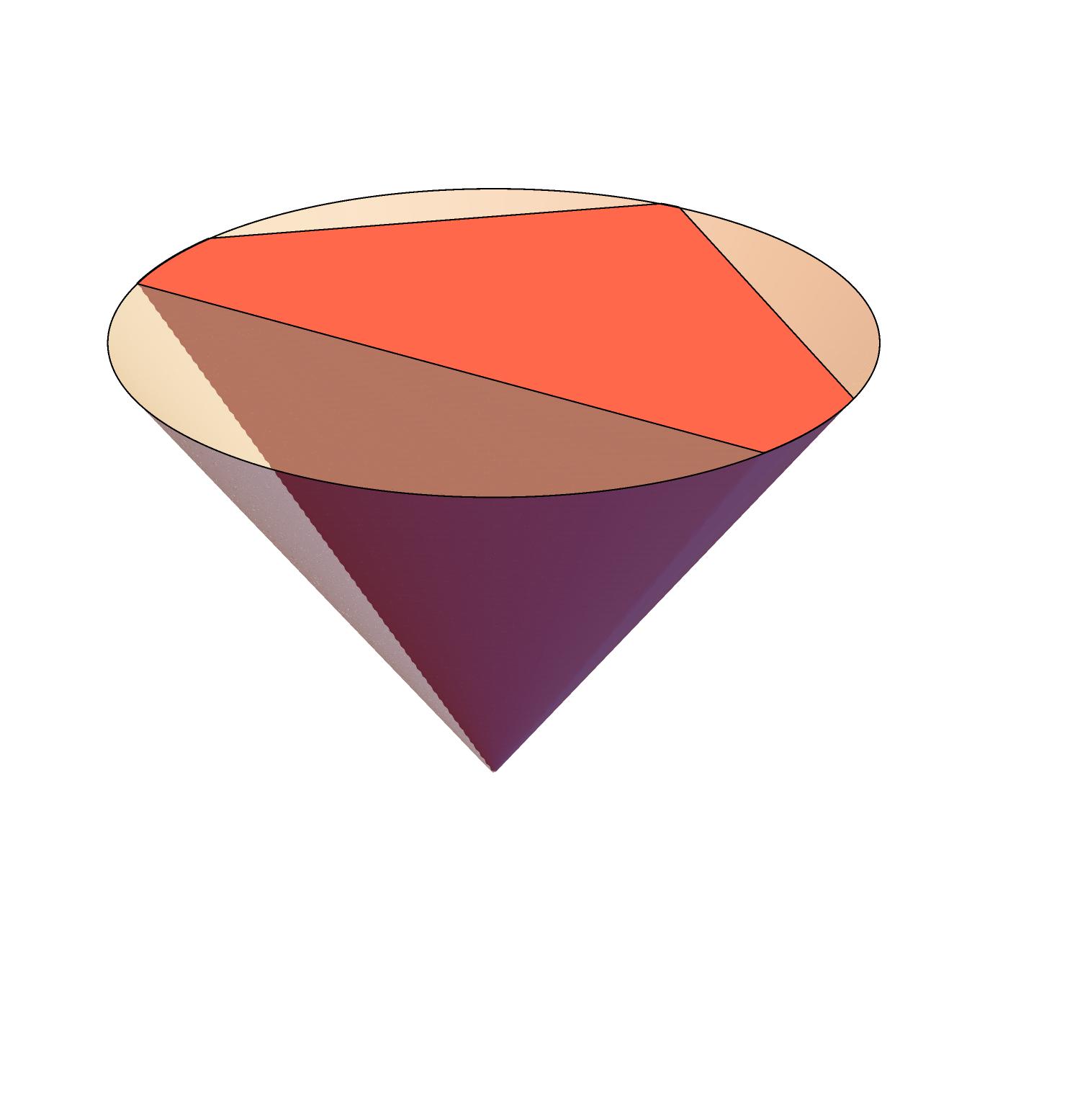}
			};
		\end{tikzpicture}
	\end{center}
	\vspace{-3.5em}
	\caption{Illustrations of \cref{lem:psd_sum_rog} (left picture) and \cref{prop:gen_psd_sum_rog_SM} (right picture) in the setting $\S^{n+1}_+=\S^2_+$. Recall that the matrices on the boundary of $\S^2_+$ all have rank at most one.}
	\label{fig:nonintersecting_ROG}
\end{figure}

\cref{lem:SM_rog_iff_nonzero_envelope,rem:NM_EXM_rel} lead to the following sufficient condition for the ROG property \cite[Theorem 1]{argue2020necessary}.
\begin{theorem}
\label{thm:ab_suffices}
Suppose $\cM = \set{\Sym(ab^\top):\, b\in \cB}$ for some $a\in\R^{n+1}$ and $\cB\subseteq\R^{n+1}$. Then, for every positive semidefinite $Z$ of rank at least two, we have $\range(Z)\cap\cN(\cM)\neq \set{0}$. In particular, $\cS(\cM)$ is ROG.
\end{theorem}

Theorem~\ref{thm:ab_suffices} has a few immediate corollaries. The first of these corollaries  allows us to handle conic constraints added to the positive semidefinite cone. Let $K$ be a closed convex cone and define $\cM \coloneqq \set{\Sym(-cb^\top):\, b\in K^*}$ where $K^*$ is the dual cone of $K$. Then, $\set{Z\in\S^{n+1}_+:\, Zc\in K} = \cS(\cM)$. We arrive at the following corollaries; see \cite[Corollaries 4 and 5]{argue2020necessary}.
\begin{corollary}\label{cor:cone_c_suffices}
Let $K\subseteq \R^{n+1}$ be any closed convex cone and consider an arbitrary vector $c\in\R^{n+1}$. Then, the set $\set{Z\in\bb S^{n+1}_+:~ Zc\in K}$ is ROG.
\end{corollary}

\begin{corollary}\label{cor:ac_bc_suffices}
Let $a,b,c\in\R^{n+1}$. Then, the set $\set{Z\in\bb S^{n+1}_+:~ a^\top Zc\geq 0,\, b^\top Z c\geq 0}$ is ROG.
\end{corollary}

As an immediate application of \cref{cor:ac_bc_suffices} and \cref{prop:epi_conv_hull_bounded_quadratic_set}, we recover the result presented in \cref{ex:big-M} on the perspective reformulation trick.

\begin{remark}
\label{lem:one_cap}
Defining $L \coloneqq \Diag(1,\dots,1,-1)\in\S^{n+1}$, we can write $\L^{n+1} = \set{z\in\R^{n+1}:\, z^\top L z\leq 0,\, z_{n+1}\geq 0}$.
\citet{sturm2003cones} (see also \cite[Section 6.1]{burer2015gentle}) established that the set 
\begin{align*}
\cS \coloneqq \set{Z\in\S^{n+1}_+:\, \begin{array}
	{l}
	Zc \in \L^{n+1}\\
	\ip{L,Z} \leq 0
\end{array}},
\end{align*}
where $c\in\R^{n+1}$, is ROG (cf.\ \cref{cor:cone_c_suffices}).
This result can also be recovered from a straightforward application of \cref{lem:M_leq_1,cor:cone_c_suffices}; see \cite[Lemma 12]{argue2020necessary}.
\end{remark}

\subsection{Necessary conditions} \label{sec:necessary_conditions}

In this section, we discuss the complete characterization of ROG cones defined by two LMIs or LMEs given in \cite[Theorem 3]{argue2020necessary}.
\begin{theorem}
\label{thm:two_LMI_NS}
Let $\cM = \set{M_1,M_2}$. Then, $\cT(\cM)$ (and thus $\cS(\cM)$) is ROG if and only if one of the following holds:
\begin{enumerate}[(i)]
	\item there exists $(\alpha_1,\alpha_2)\neq (0,0)$ such that $\alpha_1M_1+\alpha_2M_2 \in\S^{n+1}_+$, or
	\item there exists $a,b,c\in\R^{n+1}$ such that $M_1 = \Sym(ac^\top)$ and $M_2 = \Sym(bc^\top)$.
\end{enumerate}
\end{theorem}
Note that the \emph{if} direction of \cref{thm:two_LMI_NS} is a direct consequence of the sufficient conditions identified in \cref{prop:gen_psd_sum_rog_SM} and \cref{cor:ac_bc_suffices}. Furthermore, recall from \cref{cor:M2_SM_rog_iff_TM_rog} that when $\abs{\cM} =2$, the set $\cS(\cM)$ is ROG if and only if $\cT(\cM)$ is ROG. Thus it suffices to show that if $\cT(\cM)$ is ROG then one of the conditions (i) or (ii) must hold.

\begin{remark}
  \label{rem:gordan_stiemke}
  The conic Gordan--Stiemke Theorem (see Equation 2.3 in \cite{sturm2000error}
  and its surrounding comments) implies that for any subspace $W\subseteq
\S^{n+1}$,
\begin{align*}
W \cap \S^{n+1}_+ = \set{0} \iff W^\perp \cap \S^{n+1}_{++} \neq \emptyset.
\end{align*}
In particular, applying the conic Gordan--Stiemke Theorem in the context of 
\cref{thm:two_LMI_NS} we deduce that
 if $M_1,M_2$ are linearly independent, then condition (i) in \cref{thm:two_LMI_NS} 
 fails if and only if $\cT(\set{M_1,M_2})$ contains a positive definite matrix.\qedhere
\end{remark}

 Conditions (i) and (ii) in \cref{thm:two_LMI_NS} have
simple geometric interpretations.
See \cref{rem:geometric_interpretation_i} for a geometric interpretation of (i).
Condition (ii) covers the important case when the two LMEs interact in a nontrivial manner inside $\S^{n+1}_+$. Suppose for the sake of presentation that $a = e_1$, $b=e_2$, $c=e_{n+1}$. Then, \cref{cor:ac_bc_suffices} implies that
\begin{align*}
\cT(\cM) &= \conv(\set{zz^\top:\, z_1z_{n+1} = 0,\, z_2z_{n+1} = 0})\\
&= \conv\left(\conv\set{zz^\top:\, z_1 = z_2 = 0}\cup \conv\set{zz^\top:\, z_{n+1} = 0}\right)\\
&= \conv\left((0_2\oplus \S^{n-1}_+ )\cup (\S^{n}_+\oplus 0_1) \right).
\end{align*}
In other words, condition (ii) covers the case where $\cT(\cM)$ is the convex hull of the union of two faces of the positive semidefinite cone with a particular intersection structure.
\cref{thm:two_LMI_NS} states that these are the only ways for $\cT(\cM)$ to be ROG when $\abs{\cM} = 2$.

\begin{remark}\label{rem:certificates}
Both directions of \cref{thm:two_LMI_NS} admit small certificates; see \cite[Remark 11]{argue2020necessary}. Let $\cM = \set{M_1,M_2}$.
\begin{itemize}
\item Suppose $\cS(\cM)$ is ROG. Then \cref{thm:two_LMI_NS} implies that there exists either aggregation weights $(\alpha_1,\alpha_2)\neq(0,0)$ for which $\alpha_1M_1 + \alpha_2M_2 \in\S^{n+1}_+$ or vectors $a,b,c\in\R^{n+1}$ for which $M_1 = \Sym(ac^\top)$ and $M_2 = \Sym(bc^\top)$.
\item Suppose $\cS(\cM)$ is not ROG. Then, based on \cref{thm:two_LMI_NS}, it suffices to certify that neither conditions (i) nor (ii) hold.
Note that $M_1$ and $M_2$ are linearly
independent since otherwise we would have $\cS(\cM)$ is ROG. Then, from the Gordan--Stiemke Theorem (see \cref{rem:gordan_stiemke}) we deduce that condition (i) fails if and only if there exists a positive definite matrix $Z$ in $\cT(\cM)$. That is, presenting a positive definite matrix in $\cT(\cM)$ will certify that  condition (i) fails.
If either $\rank(M_1)\geq 3$ or $\rank(M_2)\geq 3$, then the spectral decomposition of the corresponding $M_i$ certifies that condition (ii) does not hold.
Else, $M_1$ and $M_2$ are both indefinite rank-two matrices and we can write $M_1 = \eta_1 \Sym(ab^\top)$ and $M_2=\eta_2 \Sym(cd^\top)$ where $\eta_i \in\R$ and $a,b,c,d\in\R^{n+1}$ are unit vectors. This decomposition is unique up to renaming $a$ and $b$ or $c$ and $d$. Then, condition (ii) does not hold if and only if $a,b,c,d$ are distinct. In particular, this decomposition certifies that condition (ii) does not hold.\qedhere
\end{itemize}
\end{remark}

The proof of \cref{thm:two_LMI_NS} is nontrivial and requires several arguments. We give a proof outline below. See \cite[Section 4]{argue2020necessary} for a more detailed proof.

\begin{proof}
[Proof outline for \cref{thm:two_LMI_NS}]
We may assume that
\begin{align*}
\spann\left(\range(M_1) \cup\range(M_2)\right) = \R^{n+1}
\end{align*}
without loss of generality.
Indeed, when the set $\set{M_1,M_2}$ does not satisfy this assumption, we may consider the set of restricted matrices $\set{(M_1)_W, (M_2)_W}$ where $(M_i)_W$ is the restriction of $M_i$ onto the minimal subspace, $W$, containing $\range(M_1) \cup\range(M_2)$. It is not hard to show that the ROG property as well as conditions (i) and (ii) are invariant under this operation (see \cite[Lemma 13]{argue2020necessary}).

By \cref{lem:psd_sum_rog,cor:ac_bc_suffices}, it suffices to show that if $\cT(\cM)$ is ROG, then either condition (i) or condition (ii) holds.
We will split the proof of \cref{thm:two_LMI_NS} into a number of cases depending on the dimension $n+1$.
\begin{itemize}
	\item The case $n + 1= 1$ holds vacuously as we can set $(\alpha_1,\alpha_2)$ to either $(1,0)$ or $(-1,0)$ to satisfy (i).
	\item For $n + 1 = 2$, we can show that condition (i) necessarily holds. Indeed, supposing otherwise, we can explicitly construct a rank-two extreme ray of $\cT(\cM)$ using the Gordan--Stiemke Theorem. The construction crucially uses the geometry of $\R^2$ (and $\S^2$); see \cite[Proposition 2]{argue2020necessary}.
	\item For $n+1=3$, when neither conditions (i) nor (ii) are satisfied, we can explicitly construct extreme rays of $\cT(\cM)$ with rank two. The construction is based on understanding what the corresponding $\cN(\cM)$ set looks like. This construction crucially uses the geometry of $\R^3$. In particular, it establishes that in this case where neither conditions (i) nor (ii) are satisfied, $\cN(\cM)$ is the union of at most four one-dimensional subspaces of $\R^3$ (see \cite[Lemma 17]{argue2020necessary}).\footnote{For readers familiar with algebraic geometry, this may be viewed as a consequence of B\'ezout's theorem.} Following this, we may then apply Dine's Theorem \cite{dines1941mapping} and \cite[Lemma 3.13]{hildebrand2016spectrahedral} to construct our desired rank-two extreme ray of $\cT(\cM)$ (see \cite[Proposition 3]{argue2020necessary}).
	\item Finally, we will reduce the case where $n+1\geq 4$ to the case where $n + 1 = 3$. Specifically, supposing that $\cT(\cM)$ is a ROG cone with $n+1\geq 4$ for which condition (i) does not hold,
	it is possible to construct a three-dimensional subspace $W$ of $\R^{n+1}$ such that the restriction of $\cM$ to $W$, denoted $\cM_W$, satisfies: $\cT(\cM_W)$ is ROG, neither conditions (i) nor (ii) hold for $\cM_W$. This gives us our desired contradiction. See \cite[Proposition 4]{argue2020necessary}.
\qedhere
\end{itemize}
\end{proof}

\begin{remark}\label{rem:stronger_necessary_condition}
	The above proof outline in fact shows something stronger than \cref{thm:two_LMI_NS}. Specifically, in the cases where
	\begin{align*}
	\dim\left(\spann\left(\range(M_1)\cup\range(M_2)\right)\right) \neq 3,
	\end{align*}
	we were able to derive contradictions by simply assuming that condition (i) did not hold.
	In other words, condition (i) itself completely characterizes the ROG property of a cone defined by two LMIs whenever the dimension of their joint span is not three-dimensional.\qedhere
\end{remark}

We close by illustrating the proof of \cref{thm:two_LMI_NS} on a prototypical example where the joint span of $M_1$ and $M_2$ is three-dimensional; see \cite[Example 2]{argue2020necessary}.
\begin{example}
Suppose $\cM = \set{M_1,M_2}$ where $M_1 = \Diag(1,-1,0)$ and $M_2 = \Diag(0,1,-1)$ so that
\begin{align*}
\cT(\cM) = \set{Z\in\S^3_+:\, Z_{1,1} = Z_{2,2} = Z_{3,3}}.
\end{align*}
Note that $\alpha_1 M_1 + \alpha_2 M_2 = \Diag(\alpha_1,\alpha_2-\alpha_1, -\alpha_2)$ is positive semidefinite if and only if $(\alpha_1,\alpha_2)=(0,0)$ so that condition (i) of \cref{thm:two_LMI_NS} is violated.
Next, we claim that condition (ii) of \cref{thm:two_LMI_NS} does not hold. Indeed, assuming condition (ii), we have that $\alpha_1 M_1 + \alpha_2 M_2 = \Sym((\alpha_1 a + \alpha_2 b) c^\top)$ has rank at most two. Observing that $2M_1 +M_2 = \Diag(2,-1,1)$ has rank three, we deduce that condition (ii) of \cref{thm:two_LMI_NS} cannot hold. 
We conclude that $\cT(\cM)$ is not ROG.

Below, we construct a rank-two extreme ray of $\cT(\cM)$. 

Note that $\cN(\cM) = \set{z\in\R^3:\, z_1^2 = z_2^2 = z_3^2}$ is a union of the four lines generated by $(1,1,1)$, $(1,1, -1)$, $(1,-1,1)$, and $(1,-1,-1)$. Then,
\begin{align*}
\cR\coloneqq \bigcup_{x,y\in\cN(\cM)}\spann(x,y)
\end{align*}
consists of all the vectors in $\R^3$ with at most two different magnitudes.

As $\cR$ is a union of finitely many planes in $\R^3$, there exists a vector $w\notin \cR$. For example, we may pick $w = (-1, 0, 1)$.
Dine's Theorem \cite{dines1941mapping} states that as condition (i) does not hold, there exists some $u\in\R^3$ such that
\begin{align*}
\begin{pmatrix}
	u^\top M_1 u\\
	u^\top M_2 u
\end{pmatrix} =
-\begin{pmatrix}
	w^\top M_1 w\\
	w^\top M_2 w
\end{pmatrix}.
\end{align*}
Indeed, $u = (1,\sqrt{2},1)$ is such a vector.
Then, $Z\coloneqq ww^\top + uu^\top$ is a rank-two matrix contained in $\cT(\cM)$. By \cref{cor:TM_rog_iff_nonzero_null}, it suffices to show that $\range(Z) \cap \cN(\cM) = \set{0}$. We will write a general element from $\range(Z)$ as
$\left( \beta - \alpha, \sqrt{2}\beta, \alpha+\beta \right)$. Then
\begin{align*}
\range(Z) \cap\cN(\cM) &= \set{\begin{pmatrix}\beta-\alpha \\\sqrt{2}\beta \\ \alpha+\beta\end{pmatrix}
	:\, \begin{array}
	{l}
	(\beta-\alpha)^2=2\beta^2=(\alpha+\beta)^2
\end{array}}.
\end{align*}
Note that $(\beta-\alpha)^2=(\alpha+\beta)^2$ implies that $\alpha\beta = 0$, whence $(\beta-\alpha)^2 = 2\beta^2$ implies that $\alpha^2 = \beta^2$.
We conclude $\range(Z)\cap\cN(\cM) = \set{0}$ and that $\cT(\cM)$ is not ROG.\qedhere
\end{example}

\subsection{Minimizing ratios of quadratic functions over ROG domains}
\label{subsec:quadratic_ratios}
In this section, we show how a ``re-homegenization'' trick can be combined with our toolset (specifically \cref{lem:SDPtightness}) to minimize the ratio of two quadratic functions over a ROG domain.
Let $M_\obj, B\in\S^{n+1}$ and let $\cM\subseteq\S^{n+1}$.
We will consider the following optimization problem:
\begin{align}
\label{eq:ratio_of_quadratics_positive}
\inf_{\tilde z\in\R^{n+1}}\set{\tfrac{\tilde z^\top M_\obj\tilde  z}{\tilde z^\top B\tilde  z}:\, \begin{array}
	{l}
	\tilde z\tilde z^\top \in\cS(\cM)\\
	\tilde z^\top B \tilde z>0\\
	\tilde z^2_{n+1} = 1
\end{array}}.
\end{align}

\begin{remark}
Note that the variant of \eqref{eq:ratio_of_quadratics_positive} where the constraint $\tilde z^\top B\tilde z>0$ is replaced with $\tilde z^\top B\tilde z\neq0$ can be decomposed as two instances of \eqref{eq:ratio_of_quadratics_positive} based on the sign of $\tilde z^\top B\tilde z$ (and negating both $M_\obj$ and $B$ on the portion of the domain where $\tilde z^\top B\tilde z$ is negative).
\end{remark}

We derive an SDP relaxation to \eqref{eq:ratio_of_quadratics_positive} as follows:
\begin{align}
\label{eq:ratio_of_quadratics_sdp1}
\inf_{\tilde z\in\R^{n+1}}\set{\tfrac{\tilde z^\top M_\obj\tilde  z}{\tilde z^\top B\tilde  z}:\, \begin{array}
	{l}
	\tilde z\tilde z^\top \in\cS(\cM)\\
	\tilde z^\top B \tilde z>0\\
	\tilde z^2_{n+1} = 1
\end{array}}
&= \inf_{z\in\R^{n+1}}\set{z^\top M_\obj z:\, \begin{array}
	{l}
	zz^\top \in\cS(\cM)\\
	z^\top B z=1\\
	z^2_{n+1} >0
\end{array}}\\
\label{eq:ratio_of_quadratics_sdp2}
&\geq \inf_{z\in\R^{n+1}}\set{z^\top M_\obj z:\, \begin{array}
	{l}
	zz^\top \in\cS(\cM)\\
	z^\top B z=1
\end{array}}\\
\label{eq:ratio_of_quadratics_sdp3}
&\geq \inf_{Z\in\S^{n+1}}\set{\ip{M_\obj, Z}:\, \begin{array}
	{l}
	Z \in\cS(\cM)\\
	\ip{B,Z} = 1
\end{array}},
\end{align}
where \eqref{eq:ratio_of_quadratics_sdp1} follows from a simple change of variables to have the desired scaling relations, \eqref{eq:ratio_of_quadratics_sdp2} is obtained by dropping the constraint $z^2_{n+1} >0$, and we dropped the rank-1 requirement on the matrix in the final relaxation step of \eqref{eq:ratio_of_quadratics_sdp3}.

\cref{lem:SDPtightness} implies that the second inequality holds with equality whenever $\cS(\cM)$ is ROG and \eqref{eq:ratio_of_quadratics_sdp3} is bounded below. This boundedness holds under relatively minor assumptions. Similarly, a variety of different assumptions may be used to guarantee that the inequality relation in  \eqref{eq:ratio_of_quadratics_sdp2} holds with equality.
The following lemma demonstrates one such pair of sufficient conditions.

\begin{lemma}
\label{lem:ratio_of_quadratics}
Let $M_\obj,B\in\S^{n+1}$ and $\cM\subseteq\S^{n+1}$. Suppose $\cS(\cM)$ is ROG, there exists $M^*\in\clcone(\cM)$ and $\lambda\in\R$ such that $M_\obj + M^* + \lambda B\succeq 0$, and
\begin{align}
\label{eq:z_closure_property}
 \cl\set{z\in\R^{n+1}:\, \begin{array}
	{l}
	zz^\top \in\cS(\cM)\\
	z^2_{n+1} >0
\end{array}} = \set{z\in\R^{n+1}:\, zz^\top \in\cS(\cM)}.
\end{align}
Then, equality holds throughout \eqref{eq:ratio_of_quadratics_sdp1} to \eqref{eq:ratio_of_quadratics_sdp3}.
\end{lemma}

\begin{example}[Regularized total least squares]
The total least squares problem (TLS) adapts least squares regression to the setting where both the independent and dependent variables may be corrupted by noise \cite{golub1999tikhonov}. A variant of the TLS, known as the regularized total least squares problem (RTLS), introduces an additional regularization constraint that protects against poorly behaved solutions which arise when the data matrix has small singular values. This regularization is well studied from both theoretical and practical points of view (see \cite{xia2015minimizing,beck2009convex} and references therein).

By eliminating variables, the RTLS can be rewritten as minimizing the ratio of a nonnegative quadratic function and a positive quadratic function over a nonempty ellipsoid (see for example \cite{golub1999tikhonov}).
In particular, the RTLS can be written in the form of \eqref{eq:ratio_of_quadratics_positive} where $M_\obj, B\in\S^{n+1}_+$ and $\abs{\cM} = 1$.
It is then straightforward to verify that the assumptions of \eqref{lem:ratio_of_quadratics} are satisfied so that the RTLS admits an exact SDP relaxation in the sense of objective value exactness.
\end{example}

\section*{Acknowledgments}
This research is supported in part by 
NSF grant CMMI 1454548 and 
ONR grant N00014-19-1-2321. 

{
\bibliographystyle{plainnat} 
\bibliography{bib.bib}
}

\end{document}